\documentclass[12pt]{amsart}

\usepackage{setspace}   

\usepackage[usenames]{color}
\usepackage{fullpage,url,mathrsfs,stmaryrd}
\usepackage{amsmath,amsthm,amssymb,mathrsfs,stmaryrd,color}
\usepackage[utf8]{inputenc}
\usepackage[T1]{fontenc}
\usepackage{eucal}

\usepackage{relsize}
\usepackage[bbgreekl]{mathbbol}
\usepackage{amsfonts}

\DeclareSymbolFontAlphabet{\mathbb}{AMSb}
\DeclareSymbolFontAlphabet{\mathbbl}{bbold}
\newcommand{\prism}{{\mathlarger{\mathbbl{\Delta}}}}
\newcommand{\prismp}{{\prism'}}
\newcommand{\prismpp}{{\prism''}}

\newcommand{\BA}{{\mathbb{A}}}
\newcommand{\BP}{{\mathbb{P}}}

\newcommand{\BG}{{\mathbb{G}}}

\usepackage{amssymb}
\usepackage[all]{xy}
\usepackage{mathrsfs}
\usepackage{enumerate}
\usepackage{amscd}

\usepackage{bbm}

\DeclareMathOperator{\Grpds}{{Grpds}}

\DeclareMathOperator{\Sm}{{Sm}}

\DeclareMathOperator{\KS}{{KS}}
\DeclareMathOperator{\pCurv}{{p-Curv}}
\DeclareMathOperator{\exact}{{exact}}
\DeclareMathOperator{\closed}{{closed}}

\DeclareMathOperator{\Ob}{{Ob}}
\DeclareMathOperator{\Sch}{{Sch}}

\DeclareMathOperator{\Stacks}{{Stacks}}
\DeclareMathOperator{\Tot}{{Tot}}

\newcommand{\AAut}{\underline{\on{Aut}}}

\DeclareMathOperator{\Shim}{{\bf Shim}}

\DeclareMathOperator{\ad}{{ad}}

\newcommand{\cA}{{\mathcal A}}
\newcommand{\cB}{{\mathcal B}}

\newcommand{\cE}{{\mathcal E}}
\newcommand{\cF}{{\mathcal F}}

\newcommand{\cN}{{\mathcal N}}
\newcommand{\cO}{{\mathcal O}}

\newcommand{\cT}{{\mathcal T}}

\newcommand{\sR}{{\mathscr R}}

\newcommand{\sX}{{\mathscr X}}
\newcommand{\sY}{{\mathscr Y}}

\newcommand{\fa}{{\mathfrak a}}

\newcommand{\fg}{{\mathfrak g}}
\newcommand{\fh}{{\mathfrak h}}

\newcommand{\nc}{\newcommand}

\nc\wh{\widehat}

\nc\on{\operatorname}

\nc\Gr{\on{Gr}}

\nc\Fl{\on{Fl}}

\newtheorem{cor}[subsubsection]{Corollary}
\newtheorem{lem}[subsubsection]{Lemma}
\newtheorem{prop}[subsubsection]{Proposition}

\newtheorem{conj}[subsubsection]{Conjecture}
\newtheorem{thm}[subsubsection]{Theorem}

\theoremstyle{remark}
\newtheorem{rem}[subsubsection]{Remark}

\newcommand{\BF}{{\mathbb{F}}}
\newcommand{\BN}{{\mathbb{N}}}

\newcommand{\BZ}{{\mathbb{Z}}}

\DeclareMathOperator{\Lie}{{Lie}}

 \DeclareMathOperator{\Spf}{{Spf}}

\DeclareMathOperator{\Cone}{{Cone}}

\DeclareMathOperator{\BT}{{BT}}
\DeclareMathOperator{\Disp}{{Disp}}
\DeclareMathOperator{\Syn}{{Syn}}
\DeclareMathOperator{\syn}{{syn}}
\DeclareMathOperator{\Lau}{{Lau}}

\newcommand{\limto}{{\displaystyle\lim_{\longrightarrow}}}
\newcommand{\rightlim}{\mathop{\limto}}

\newcommand{\leftlim}{\mathop{\displaystyle\lim_{\longleftarrow}}}
\newcommand{\limfromn}{\leftlim\limits_{\raise3pt\hbox{$n$}}}
\newcommand{\limton}{\rightlim\limits_{\raise3pt\hbox{$n$}}}

\newcommand{\rightlimit}[1]{\mathop{\lim\limits_{\longrightarrow}}\limits%
                    _{\raise3pt\hbox{$\scriptstyle #1$}}}

\newcommand{\leftlimit}[1]{\mathop{\lim\limits_{\longleftarrow}}\limits%
                    _{\raise3pt\hbox{$\scriptstyle #1$}}}

\newcommand{\epi}{\twoheadrightarrow}
\newcommand{\iso}{\buildrel{\sim}\over{\longrightarrow}}
\newcommand{\mono}{\hookrightarrow}

\DeclareMathOperator{\dR}{{dR}}
\DeclareMathOperator{\Hodge}{{Hodge}}

\DeclareMathOperator{\et}{{et}} 
\DeclareMathOperator{\fppf}{{fppf}} 
 
\DeclareMathOperator{\Hom}{{Hom}}

\DeclareMathOperator{\Ker}{{Ker}} \DeclareMathOperator{\id}{{id}}
\DeclareMathOperator{\im}{{Im}}

\DeclareMathOperator{\Mor}{{Mor}}
 \DeclareMathOperator{\op}{{op}}

\DeclareMathOperator{\Spec}{{Spec}}

\DeclareMathOperator{\WCart}{{WCart}}

\theoremstyle{definition}

\newtheorem{ex}[subsubsection]{Example}

\numberwithin{equation}{section}

\newcommand{\Fr}{\operatorname{Fr}}

\begin{document}
\title[Shimurian generalizations of $\BT_1\otimes\BF_p$]{On Shimurian generalizations of the stack $\BT_1\otimes\BF_p$}
\author{Vladimir Drinfeld}
\address{University of Chicago, Department of Mathematics, Chicago, IL 60637}

\begin{abstract}
Let $G$ be a smooth group scheme over $\BF_p$ equipped with a $\BG_m$-action such that all weights of $\BG_m$ on $\Lie (G)$ are $\le 1$.
 Let $\Disp_n^G$ be Eike Lau's stack of $n$-truncated $G$-displays (this is an algebraic $\BF_p$-stack).  
 In the case $n=1$ we introduce an algebraic stack equipped with a morphism to $\Disp_1^G$. We conjecture that if $G=GL(d)$ then the new stack is  canonically isomorphic to the reduction modulo $p$ of the stack of $1$-truncated Barsotti-Tate groups of height $d$ and  dimension $d'$, where $d'$ depends on the action of $\BG_m$ on $GL(d)$.
 
We also discuss how to define an analog of the new stack for $n>1$ and how to replace $\BF_p$ by $\BZ/p^m\BZ$.
\end{abstract}


\maketitle

\section{Introduction}
Once and for all, we fix a prime $p$. The words ``algebraic stack'' are understood in the sense of \cite{LM} unless stated otherwise.
 
\subsection{The goal}    \label{ss:The goal}
The notion of $n$-truncated Barsotti-Tate group was introduced by Grothen\-dieck \cite{Gr}. It is reviewed in \cite{Me72,Il,dJ}.
$n$-truncated Barsotti-Tate groups of height $d$ and dimension $d'\le d$ form an algebraic stack over $\BZ$, denoted by $\BT_n^{d,d'}$.
This stack is smooth by a deep theorem of Grothendieck, whose proof is given in Illusie's article~\cite{Il}.

The stack $\BT_n^{d,d'}$ is related to the group $G=GL(d)$; e.g., 
$\BT_n^{d,d'}\otimes\BZ [p^{-1}]$ is the classifying stack of $GL(d, \BZ/p^n\BZ)$.
Since there exist Shimura varieties of type $E_7$, it is natural to expect that the stacks $\BT_n^{d,d'}\otimes \BZ/p^m\BZ$ (and maybe the stacks $\BT_n^{d,d'}$ themselves) have interesting $E_7$-analogs.

In the main body of this article we focus on the case $n=m=1$. In this case we give a rather elementary definition of the stack $\BT_{1,\BF_p}^G=\BT_1^G\otimes\BF_p$, where $G$ is any smooth affine group scheme over $\BF_p$ equipped with a 1-bounded $\BG_m$-action (1-boundedness means that the weights of $\BG_m$ on $\Lie (G)$ are $\le 1$). Note that a semisimple group of type $E_7$ can be equipped with a nontrivial 1-bounded $\BG_m$-action.
In Appendix~\ref{s:BT_n^G via syntomification} we suggest a definition of $\BT_{n,\BZ/p^m\BZ}^G$ for arbitrary $n$ and $m$, which uses the prismatic theory.\footnote{In the case $m=1$ one only needs the crystalline theory, and in the case $n=m=1$ the de Rham theory is enough. This is why in the case $m=n=1$ an elementary approach is possible.}

We prove that the stack $\BT_{1,\BF_p}^G$ defined in \S\ref{ss:the key definition} is algebraic (this is not an immediate consequence of the definition). 
We conjecture that in the case $G=GL(d)$ one has $\BT_{1,\BF_p}^G=\BT_1^{d,d'}\otimes\BF_p$, where $d'$ depends on the action of $\BG_m$ on $GL(d)$. The precise formulations of this result and conjecture are given in \S\ref{ss:main theorems}-\ref{ss:conjecture for G=GL(d)}.

\subsection{Structure of the article}
In \S\ref{s:Shim_n} we recall some facts about group schemes equipped with a $\BG_m$-action. We also introduce a (2,1)-category $\Shim_n$, whose objects are smooth affine group schemes over $\BZ/p^n\BZ$ equipped with a $1$-bounded $\BG_m$-action. 

In \S\ref{s:Disp} we recall a certain $\BF_p$-stack $\Disp_1^G$, $G\in\Shim_1$, which is called the stack of 1-truncated displays (or the stack of $F$-zips).
We also introduce in \S\ref{ss:Lau_1} a commutative group scheme over $\Disp_1^G$ denoted by $\Lau_1^G$.

In \S\ref{s:the key definition} we define an $\BF_p$-stack $\BT_{1,\BF_p}^G$, which depends functorially on $G\in\Shim_1$.
We also formulate the main theorems and a conjecture, see \S\ref{ss:main theorems}-\ref{ss:conjecture for G=GL(d)}. 

The proof of the main theorems is given in \S\ref{s:the conditions for nabla}-\ref{s:proof of main theorems}.

In Appendices~\ref{s:stacky nonsense}-\ref{s:gerbe generalities} we discuss some material on algebraic stacks.

In Appendix~\ref{s:BT_n^G via syntomification} we suggest a definition of $\BT_{n,\BZ/p^m\BZ}^G$ for arbitrary $n$ and $m$, which uses the prismatic theory. We also formulate Conjecture~\ref{conj:algebraicity}; it is motivated by Corollary~\ref{c:smooth algebraic} and by Grothendieck's smoothness theorem mentioned at the beginning of \S\ref{ss:The goal}.

In Appendix~\ref{s:comparing the definitions of BT_1} we explain why the definition of $\BT_{1,\BF_p}^G$ given in Appendix~\ref{s:BT_n^G via syntomification} is equivalent to the one from \S\ref{s:the key definition}.

\subsection{Status of the conjectures}
As far as I understand, both conjectures formulated in this article are proved in \cite{GMM}.

\subsection{Acknowledgements}
I thank Alexander Beilinson, Bhargav Bhatt, Akhil Mathew, Nicholas Katz, Martin Olsson, George Pappas, and Vadim Vologodsky for valuable discussions and references.
The author's work on this project was partially supported by NSF grant DMS-2001425.

\section{The (2,1)-category $\Shim_n$}   \label{s:Shim_n}
\subsection{1-bounded $\BG_m$-actions}   \label{ss:1-boundedness}
Let $G$ be a smooth group scheme over $\BZ/p^n\BZ$ and $\fg:=\Lie (G)$. An action of $\BG_m$ on $G$ induces its action on $\fg$
and therefore a grading 
\begin{equation}  \label{e:2eigenspaces}
\fg=\bigoplus_{i\in\BZ}\fg_i \, . 
\end{equation}
Following E.~Lau \cite{Lau21}, we say that an action of $\BG_m$ on $G$ is \emph{1-bounded} if $\fg_i=0$ for $i>1$.

\begin{ex}  \label{ex:reductive}
If $G$ is reductive and connected then an action of $\BG_m$ on $G$ is given by a cocharacter $\mu:\BG_m\to G_{\ad}$. In this case 1-boundedness implies that all weights of $\BG_m$ on $\Lie (G)$ belong to $\{-1,0,1\}$. Cocharacters $\mu$ with this property are called \emph{minuscule}\footnote{This notion of minuscule cocharacter is slightly more general than that of Bourbaki (e.g., Bourbaki requires a minuscule cocharacter to be nonzero).}.
\end{ex}

\subsection{The (2,1)-category $\Shim_n$}  \label{ss:Shim_n}
Let $\Shim_n$ be the following (2,1)-category: its objects are smooth affine group schemes over $\BZ/p^n\BZ$ equipped with a $1$-bounded $\BG_m$-action, and for $G_1,G_2\in\Ob\Shim_n$, the groupoid of morphisms $G_1\to G_2$ is the quotient groupoid\footnote{If a group $\Gamma$ acts on a set $X$ then the \emph{quotient groupoid} is defined as follows: the set of objects is $X$, a morphism $x\to x'$ is an element $\gamma\in\Gamma$ such that $\gamma x=x'$, and the composition of morphisms is given by multiplication in $G$.
} of $\Hom^{\BG_m}(G_1,G_2)$ by the conjugation action of $G_2^{\BG_m}(\BZ/p^n\BZ )$. Here $\Hom^{\BG_m}$ is the set of $\BG_m$-equivariant homomorphisms, and $G_2^{\BG_m}\subset G_2$ is the group subscheme of $\BG_m$-fixed points.

The above (2,1)-category $\Shim_n$ is not quite good from a certain viewpoint\footnote{See \S 9.1.4 of version 1 of \cite{GMM}.}, but we work with it in this article.

\subsection{The subgroups $M$, $P^\pm$, $U^\pm$}   \label{ss:parabolics,Levi}
\subsubsection{Recollections from \cite{CGP}}
Let $k$ be a ring and $G$ a smooth affine group $k$-scheme equipped with a $\BG_m$-action, i.e., a homomorphism $\mu :\BG_m\to\AAut\, G$. Let $M:=G^{\BG_m}$, i.e., $M\subset G$ is the subgroup of $\BG_m$-fixed points. Let $P^+$ be the attractor for $\mu$, i.e., $P^+\subset G$ is the maximal closed subscheme such that the action map $\BG_m\times P^+\to G$ extends to a morphism of schemes 
$\BA^1\times P^+\to G$. Restricting the latter to $\{ 0\}\times P^+\subset\BA^1\times P^+$, we get a retraction $P^+\to M\subset P^+$. Let $P^-\subset G$ be the attractor for $\mu^{-1}$; we have a retraction $P^-\to M\subset P^-$. Let $U^\pm$ be the preimage of $1$ with respect to the map $P^\pm\to M$. The next lemma is well known if $G$ is reductive and connected (in this case $P^+$ and $P^-$ are parabolics, and $M$ is a Levi).

\begin{lem}
(i) $P^\pm$ and $U^\pm$ are group subschemes of $G$. The maps $P^\pm\to M$ are group homomorphisms. One has $P^\pm=M\ltimes U^\pm$.

(ii) $P^+\cap P^-=M$.

(iii) $P^\pm$, $U^\pm$, and $M$ are smooth. In terms of the grading \eqref{e:2eigenspaces}, one has 
\[
\Lie (M)=\fg_0,\, \Lie (P^+)=\fg_{\ge 0},\, \Lie (P^-)=\fg_{\le 0}, \,\Lie (U^+)=\fg_{> 0},\, \Lie (U^-)=\fg_{< 0},
\]
where $\fg_{\ge 0}:=\bigoplus\limits_{i\ge 0}\fg_i$ and $\fg_{\le 0}, \fg_{> 0},\fg_{< 0}$ are defined similarly.
 
(iv) The multiplication map $U^-\times M\times U^+\to G$ is an open immersion.

(v) The fibers of $U^\pm$ over $\Spec k$ are connected. If the fibers of $G$ over $\Spec k$ are connected then the same holds for $P^\pm$ and $M$.
\end{lem}

\begin{proof}
Apply the results from \cite[p.~48-56]{CGP} (especially  \cite[Prop.~2.1.8]{CGP}) to the semidirect product $\BG_m\ltimes G$.
\end{proof}

\subsubsection{The 1-bounded case}   \label{sss:U^+=g_1}
Now assume that $\fg_i=0$ for $i>1$. Then $\fg_1$ is abelian. Moreover, by  \cite[Lemma 6.3.2]{Lau21},
there is a unique $\BG_m$-equivariant isomorphism of group schemes
\begin{equation}   \label{e:U^+=g_1}
f:U^+\iso\fg_1
\end{equation}
such that $\Lie (f)=\id_{\fg_1}$. If the ring $k$ is local then the projective $k$-module $\fg_1$ is free, so the isomorphism \eqref{e:U^+=g_1} implies that $U^+$ is 
isomorphic to a product of several copies of $\BG_a$.

If $k$ is an $\BF_p$-algebra then $\fg$ is equipped with a $p$-operation $x\mapsto x^{(p)}$. By \eqref{e:U^+=g_1}, if $\fg$ is 1-bounded then
\begin{equation}  \label{e:zero p-operation}
x^{(p)}=0 \mbox{ for } x\in\fg_1\, .
\end{equation} 
Another way to prove \eqref{e:zero p-operation} is to note that more generally, the $p$-operation takes $\fg_i$ to $\fg_{pi}$.
This follows from the fact that for any $k$-algebra $\tilde k$ and any $a\in\tilde k^\times$ the action of $a$ on $\fg\otimes_k\tilde k$ preserves the structure of restricted Lie algebra on $\fg\otimes_k\tilde k$.

\section{Recollections on the stacks $\Disp_1^G$} \label{s:Disp}
\subsection{Plan of this section}
For any $n\in\BN$ and a smooth group scheme $G$ over $\BZ/p^n\BZ$ equipped with a $\BG_m$-action, one has the \emph{$\BF_p$-stack of $n$-truncated displays} $\Disp_n^G$ defined\footnote{We are talking about the particular case of a general construction of \cite{Lau21} corresponding to the ``$n$-truncated Witt frame'' in the sense of \cite[Example~2.1.6]{Lau21}.}
in~\cite{Lau21}. 

In \S\ref{ss:Disp} we recall the stack $\Disp_1^G$. Assuming that $G\in\Shim_1$, we define in \S\ref{ss:Lau_1} a commutative group scheme on $\Disp_1^G$; we denote it by $\Lau_1^G$ (because it is conjecturally related to E.~Lau's work \cite{Lau13}, see \S\ref{ss:conjecture for G=GL(d)} below). 

The stack $\BT_{1,\BF_p}^G$, which will be defined in \S\ref{s:the key definition} assuming that $G\in\Shim_1$, will be equipped with a map to $\Disp_1^G$, which makes it into a gerbe over $\Disp_1^G$ banded by $\Lau_1^G$.

Let us make some historical remarks related to $\Disp_1^G$. First, as noted in Example~3.7.5 of \cite{Lau21}, $\Disp_1^G$ is the same as the stack of \emph{$F$-zips} in the sense of \cite{MW, PWZ}.
Second, the projective limit of $\Disp_n^G$ is the reduction modulo $p$ of a stack over $\Spf\BZ_p$, which was defined in \cite{BP}  (at least, assuming that $G$ is reductive and the cocharacter $\mu :\BG_m\to G_{\ad}$ is minuscule) following the ideas of Thomas Zink. The articles \cite{MW, PWZ, BP} preceded \cite{Lau21}.

\subsection{The $\BF_p$-stack $\Disp_1^G$}    \label{ss:Disp}
\subsubsection{Definition of the stack $\Disp_1^G$}  \label{sss:Disp definition}
Let $G$ be a smooth group scheme over $\BF_p$ equipped with a $\BG_m$-action. 
Let $M$ and $P^\pm$ be as in \S\ref{ss:parabolics,Levi}. Let $S$ be an $\BF_p$-scheme. Then $\Disp_1^G (S)$ is the groupoid of the following data:
 
 (i) a $P^\pm$-torsor $\cF^\pm$ on $S$;
 
 (ii) an isomorphism between the $G$-torsors corresponding to $\cF^+$ and $\cF^-$;
  
 (iii) an isomorphism $\cF^+_M\iso\Fr_S^*\cF^-_M$, where $\cF^\pm_M$ is the $M$-torsor corresponding to $\cF^\pm$.
 
\subsubsection{$\Disp_1^G$ as a quotient stack}  \label{sss:Disp_1^G as a quotient}
 As noted in \cite{PWZ}, the stack $\Disp_1^G$ has an explicit realization as a quotient, which shows that $\Disp_1^G$ is a quasi-compact smooth alegbraic stack over $\BF_p$ of pure dimension~$0$ with affine diagonal. To get this realization, note that the combination of data (i) and (iii) is the same as a principal $K$-bundle $\cE\to S$, where $K\subset G\times G$ is the following subgroup:
\begin{equation}    \label{e:definition of K}
K:=\{ (g,h)\in P_+\times P_-\,|\, g_M=\Fr (h_M)\},
\end{equation}
(here $g_M ,h_M$ are the images of $g,h$ in $M$). Moreover, data (ii) is the same as a $K$-equivariant morphism $\cE\to G$, where $(g,h)\in K$ acts on $G$ by 
\begin{equation}   \label{e:action of K on G}
x\mapsto hxg^{-1}.
\end{equation}
Thus $\Disp_1^G$ identifies with the quotient of the scheme $G$ by the action of $K$ given by \eqref{e:action of K on G}.

\subsubsection{The generic locus of $\Disp_1^G$} \label{sss:generic locus}
By \eqref{e:definition of K}, the stabilizer of $1\in G$ under the $K$-action given by \eqref{e:action of K on G} equals the finite group $M(\BF_p)$. Moreover, the $K$-orbit of $1$ is open.

By \S\ref{sss:Disp_1^G as a quotient}, $\Disp_1^G$ is a quotient of the scheme $G$. The composite morphism 
\begin{equation}   \label{e:generic point}
\Spec\BF_p\overset{1}\longrightarrow G\epi\Disp_1^G
\end{equation}
will be called the \emph{generic point} of the stack $\Disp_1^G$. The group scheme of automorphisms of this point equals $M(\BF_p)$, and the image of \eqref{e:generic point} is an open substack of $\Disp_1^G$, which identifies with the classifying stack of $M(\BF_p)$. This open substack will be called the \emph{generic locus} of $\Disp_1^G$ and denoted by $\Disp_{1, \rm gen}^G$. The word ``generic'' is justified at least if $G$ is connected (which implies that $\Disp_1^G$ is irreducible).

\subsubsection{Motivation}   \label{sss:de Rham cohomology}
A pair $(\cF^+,\cF^-)$ equipped with data (ii) from \S\ref{sss:Disp definition} is the same as a $G$-torsor $\cF$ equipped with a $P^+$-structure and a $P^-$-structure. If $G=GL(d)$ this means that the vector bundle $\xi$ corresponding to $\cF$ is equipped with two filtrations. As explained at the beginning of \cite{MW}, the typical example is when $\xi$ is the $m$-th relative de Rham cohomology of a smooth proper scheme $X$ over $S$ satisfying certain conditions\footnote{The conditions for the smooth proper morphism $f:X\to S$ are as follows:
the sheaves $R^bf_*\Omega^a_{X/S}$ should be locally free, and the Hodge-de Rham spectral sequence for $f$ should degenerate at the $E_1$ page.}.
In this situation the $P^-$-structure corresponds to the Hodge filration, and the $P^+$-structure corresponds to the conjugate one.

Note that the Gauss-Manin connection on $\xi$ is not included into the definition of $\Disp_1^G$. However, it is included into the definition of $\BT_{1,\BF_p}^G$, which will be given in \S\ref{ss:the key definition}.

\subsubsection{Reductive case}   \label{sss:Reductive case}
If $G$ is reductive it is known that the set of points of $\Disp_1^G$ is finite; moreover, Theorem~1.6 of \cite{PWZ} gives a complete description of this set (together with the topology on it),  and Theorem~1.7 of \cite{PWZ} describes the automorphism group scheme of each point.

\subsection{The group scheme $\Lau_1^G$}   \label{ss:Lau_1}
\subsubsection{Definition of $\Lau_1^G$}     \label{sss:Lau_1}
From now on, we assume\footnote{Without this assumption, the definition given below formally makes sense (but is probably useless) if one replaces $\fg_1$ by $\fg_{>0}\,$.} that $G\in\Shim_1$ (i.e., the action of $\BG_m$ on $G$ is 1-bounded).
Given an $\BF_p$-scheme $S$ and an $S$-point of $\Disp_1^G$, we will construct a group scheme over $S$. Let us use the notation of
\S\ref{sss:Disp definition}. Note that both $P^+$ and $P^-$ act on $\fg_1$ via $M$; we have a $P^+$-equivariant monomorphism $\fg_1\mono\fg$ and a
 $P^-$-equivariant epimorphism $\fg\epi\fg/\fg_{\le 0}=\fg_1$. Let $(\fg_1)_{\cF^\pm}$ be the vector bundle on $S$ corresponding to $\cF^\pm$ and the $P^\pm$-module $\fg_1$.
 We have the diagram
 \begin{equation}    \label{e:p-operation2}
 \Fr_S^*(\fg_1)_{\cF^-}\iso (\fg_1)_{\cF^+}\mono \fg_{\cF^+}\iso\fg_{\cF^-}\epi (\fg_1)_{\cF^-}
 \end{equation}
 in which the first isomorphism comes from data (iii) of \S\ref{sss:Disp definition} and the second one from~(ii). Consider $(\fg_1)_{\cF^-}$ as a commutative restricted Lie $\cO_S$-algebra with $p$-operation\footnote{This $p$-operation is \emph{different} from the one induced by the $p$-operation on $\fg_1$; the latter is zero by formula~\eqref{e:zero p-operation}.} \eqref{e:p-operation2}. This restricted Lie $\cO_S$-algebra yields a finite locally free group $S$-scheme of height 1 in the usual way (see \cite[Exp.VIIA]{SGA3} or \cite[\S 2]{dJ}). We denote it by $\Lau_{1,S}^G$. The formation of $\Lau_{1,S}^G$ commutes with base change $S'\to S$. As $S$ varies, the group schemes $\Lau_{1,S}^G$ define a commutative finite flat group scheme over $\Disp_1^G$ of height $1$, which we denote by $\Lau_1^G$.
 
 The stack $\BT_{1,\BF_p}^G$, which will be defined in \S\ref{s:the key definition}, will turn out to be a gerbe over $\Disp_1^G$ banded by $\Lau_1^G$, see Theorem~\ref{t:2}.

\subsubsection{Restriction of $\Lau_1^G$ to the generic locus of $\Disp_1^G$}  \label{sss:Lau_1^G on generic locus}
Let $\Lau_{1, \rm gen}^G$ be the fiber of $\Lau_1^G$ over the generic point of $\Disp_1^G$ in the sense of \S\ref{sss:generic locus}, 
so $\Lau_{1, \rm gen}^G$ is a finite group scheme over $\BF_p$ equipped with an action of $M(\BF_p)$. Then one has a canonical $M(\BF_p)$-equivariant isomorphism
\begin{equation}   \label{e:generic fiber of Lau_n^true}
\Lau_{1, \rm gen}^G\iso\fg_1\otimes_{\BF_p}\mu_p
\end{equation}
(indeed, in the case of the generic point of $\Disp_1^G$ the torsors $\cF^\pm$ are canonically trivial, and the composite map \eqref{e:p-operation2} is $\id_{\fg_1}$).
Less canonically, $\Lau_{1, \rm gen}^G$ is a direct sum of $\dim\fg_1$ copies of $\mu_p$.

 \section{Definition of $\BT_{1,\BF_p}^G$}   \label{s:the key definition}
\subsection{General remarks on smooth algebraic stacks}  \label{ss:smooth algebraic stacks}
\subsubsection{An easy lemma}   \label{sss:easy lemma}
Let $\Sch_{\BF_p}$ be the category of $\BF_p$-schemes and $\Sm_{\BF_p}$ the full subcategory of smooth $\BF_p$-schemes; unless specified otherwise, we will specify each of these categories with the etale topology. Let $\Grpds$ be the (2,1)-category of groupoids. 

Let $\sX$ be an $\BF_p$-stack, i.e., a functor $\Sch_{\BF_p}^{\op}\to\Grpds$ satisfying the sheaf property. Then the restriction of $\sX:\Sch_{\BF_p}^{\op}\to\Grpds$ to $\Sm_{\BF_p}^{\op}$ will be denoted by $\sX_{\Sm}$. The following variant of Yoneda's lemma is easy and well known.

\begin{lem}   \label{l:easy lemma}
The functor $\sX\mapsto\sX_{\Sm}$ becomes fully faithful when restricted to the (2.1)-category of {\bf smooth} algebraic stacks over $\BF_p$.
\end{lem}

This variant of Yoneda's lemma is probably well known. For completeness, we give a proof of a more general statement in Proposition~\ref{p:stacky Yoneda} of Appendix~\ref{s:stacky nonsense}.

\begin{rem}     \label{r:sheafified Kan extension} 
One can ask how to reconstruct a smooth algebraic stack $\sX$ over $\BF_p$ in terms of $\sX_{\Sm}$. Abstract nonsense gives the following ``answer'' (see Proposition~\ref{p:sheafified Kan extension} and the sentence after it):  
$\sX$ is the sheafified left Kan extension of $\sX_{\Sm}:\Sm_{\BF_p}^{\op}\to\Grpds$ to $\Sch_{\BF_p}^{\op}$.
\end{rem}

\subsubsection{Example}   \label{sss:alpha_p-torsors}
In this article we will be describing smooth algebraic $\BF_p$-stacks by specifying their restriction to $\Sm_{\BF_p}^{\op}$. Here is a baby example.

Let $\sX$ be the classifying stack of the group scheme 
\begin{equation}   \label{e:alpha_p}
\alpha_p:=\Ker (\Fr :\BG_a\to\BG_a), \mbox{ where } \BG_a:=(\BG_a)_{\BF_p};
\end{equation}
in other words, $\sX (S)$ is the groupoid of $\alpha_p$-torsors on $S$ for the fppf topology.
The stack $\sX$ is smooth: indeed, the map
\[
\BG_a\times\BA^1\to\BA^1 , \quad (a,x)\mapsto x+a^p.
\]
defines an action of $\BG_a$ on $\BA^1$ such that the quotient stack is $\sX$. We claim that $\sX_{\Sm}$ is the following sheaf of \emph{sets} (rather than groupoids):
\[
S\mapsto H^0(S,\Omega^1_{S, \exact}),  \mbox{ where } S\in\Sm_{\BF_p} \mbox{ and } \Omega^1_{S, \exact}:=\im(\cO_{S,\et}\overset{d}\longrightarrow\Omega^1_{S,\et}).
\]
Indeed, by \eqref{e:alpha_p},  the derived direct image of $\alpha_p$ under the morphism $S_{\fppf}\to S_{\et}$ equals $\cA [-1]$, where $\cA:=\Cone (\Fr :\cO_{S,\et}\to\cO_{S,\et})$. Finally,  smoothness of $S$ implies that $\cA=\Omega^1_{S, \exact}\,$.

\subsection{Definition of $\BT_1^G(S)$, where $S\in\Sm_{\BF_p}$}  \label{ss:the key definition}
We will use the notion of $p$-curvature of an integrable connection (see \cite[\S 5]{K70}). We say that a connection is \emph{$p$-integrable} if it is integrable and its $p$-curvature is zero.
\subsubsection{Definition}   \label{sss:the key definition}
Let $G\in\Shim_1$. Let $M$ and $P^\pm$ be as in \S\ref{ss:parabolics,Levi}. Let $S\in\Sm_{\BF_p}$. Then $\BT_1^G (S)$ is the groupoid of the following data:
 
 (i) a $P^\pm$-torsor $\cF^\pm$ on $S$;
 
 (ii) an isomorphism $\cF^+_G\iso\cF^-_G$, where $\cF^\pm_G$ is the $G$-torsor corresponding to $\cF^\pm$;
 
 (iii) an isomorphism 
 \begin{equation}   \label{e:relation between the M-bundles}
 \cF^+_M\iso\Fr_S^*\cF^-_M, 
  \end{equation}
 where $\cF^\pm_M$ is the $M$-torsor corresponding to $\cF^\pm$;
 
 (iv) an integrable connection $\nabla$ on $\cF^+$ satisfying the following conditions: first, the corresponding connection on $\cF^+_M$ should equal the one that
 comes from \eqref{e:relation between the M-bundles} and the usual connection on a $\Fr_S$-pullback; second, the following \emph{Katz condition} should hold:
 \begin{equation}   \label{e:Katz condition}
\pCurv_\nabla=-\KS_\nabla ,
 \end{equation}
 where $\pCurv_\nabla$ is the $p$-curvature of $\nabla$ and $\KS_\nabla\in H^0(S,(\fg_1)_{\cF^-}\otimes\Omega^1_S)$  is the 
 \emph{Kodaira-Spencer}\footnote{The terminology is motivated by the picture from \S\ref{sss:de Rham cohomology}. One can think of data (i)-(ii) as a $G$-bundle on $S$ equipped with a $P^+$-structure and a $P^-$-structure. Informally, these are the conjugate filtration and Hodge filtration, respectively. Also informally, we think of $\nabla$ as a Gauss-Manin connection.} 1-form defined in \S\ref{sss:defining KS} below.
 
 \medskip
 
 Let us explain why \eqref{e:Katz condition} makes sense. The right-hand side  of \eqref{e:Katz condition} is a section of the sheaf $(\fg_1)_{\cF^-}\otimes\Omega^1_S$.
 By the first condition from (iv), the connection on $\cF^+_M$ induced by $\nabla$ is $p$-integrable, so $\pCurv_\nabla$ is a  section of $(\Fr_*(\fg_1)_{\cF^+})^\nabla\otimes\Omega^1_S$, where $\Fr:=\Fr_S$ and $(\Fr_*(\fg_1)_{\cF^+})^\nabla$ is the horizontal part of $\Fr_*(\fg_1)_{\cF^+}$. 
 But \eqref{e:relation between the M-bundles} induces an isomorphism $\Fr^*(\fg_1)_{\cF^-}\iso (\fg_1)_{\cF^+}$, so 
 $(\Fr_*(\fg_1)_{\cF^+})^\nabla=(\fg_1)_{\cF^-}$ and \eqref{e:Katz condition} makes sense.

  \subsubsection{The Kodaira-Spencer 1-form}  \label{sss:defining KS}
  By \S\ref{sss:the key definition}(ii), $\cF^+$ and $\cF^-$ induce the same $G$-torsor, which we denote by $\cF$. The connection $\nabla$ on $\cF^+$ induces a connection $\nabla_G$ on $\cF$, and 
  $\KS_\nabla$ ``measures'' the failure of $\nabla_G$ to preserve the $P^-$-structure on $\cF$. More precisely, $\nabla_G$ is a section of the Atiyah extension\footnote{Recall that in terms of the principal $G$-bundle $E\to S$ corresponding to $\cF$, the sheaf $\cA$ from \eqref{e:Atiyah extension} is the sheaf of $G$-equivariant vector fields on $E$.}
 \begin{equation}   \label{e:Atiyah extension}
  0\to\fg_\cF\to\cA\to\Theta_S\to 0, \quad \Theta_S:=(\Omega^1_S)^*,
 \end{equation}
 and $\KS_\nabla :\Theta_S\to (\fg/\fg_{\le 0})_{\cF^-}$ is the composition of $\nabla_G:\Theta_S\to\cA$ and the map 
 \[
 \cA\to (\fg/\Lie(P^-))_{\cF^-}=(\fg/\fg_{\le 0})_{\cF^-}
 \]
 that comes from the $P^-$-structure on $\cF$. Note that $\fg/\fg_{\le 0}=\fg_1$ by the 1-boundedness assumption.
 
 \subsubsection{Remark}  \label{sss:without 1-boundedness}
 The above definition of $\BT_1^G(S)$ makes sense without the 1-bounded\-ness assumption if \eqref{e:Katz condition} is understood as an equality in $H^0(S,(\fg/\fg_{\le 0})_{\cF^-}\otimes\Omega^1_S)$; this equality implies the Griffiths transversality condition
 $\KS_\nabla\in H^0(S,(\fg_1)_{\cF^-}\otimes\Omega^1_S)$ because $\pCurv_\nabla$ is in $H^0(S,(\fg_1)_{\cF^-}\otimes\Omega^1_S)$. However, the 1-bounded\-ness assumption is used in the proof of the theorems formulated in \S\ref{ss:main theorems} below.

 \subsubsection{On the Katz condition}   \label{sss:Katz condition}
Condition \eqref{e:Katz condition} is inspired by Theorem~3.2 of \cite{K72}. As far as I understand, the minus sign in \eqref{e:Katz condition} does not agree\footnote{Theorem~3.2 of \cite{K72} involves $(-1)^{b+1}$, where $b$ is the number of the cohomology group. The case relevant for us is $b=1$.} with \cite[Thm.~3.2]{K72}, but it agrees with Remark~3.20 of \cite{OV}, which explains a modern point of view on Theorem~3.2 of \cite{K72}. (In \cite[Remark~3.20]{OV} the sign appears when one computes explicitly the functor $C^\bullet_{Y/S}\,$, which occurs in the l.h.s of (3.19.2) and (3.19.3).)

\subsubsection{Remark}   \label{sss:forgetting nabla}
Forgetting the connection $\nabla$ from \S\ref{sss:the key definition}(iv), one gets for each $S\in\Sm_{\BF_p}$ a functor $\BT_1^G(S)\to\Disp_1^G(S)$, where $\Disp_1^G$ is as in \S\ref{sss:Disp definition}.

\subsection{Functoriality in $G\in\Shim_1$}  \label{ss:Functoriality in G}
It is clear that the functor
\begin{equation}   \label{e:BT_1^G on Sm}
\Sm_{\BF_p}^{\op}\to\Grpds, \quad S\mapsto\BT_1^G(S).
\end{equation}
from \S\ref{sss:the key definition} depends functorially on $G\in\Shim_1$, where $\Shim_1$ is the (2,1)-category from~\S\ref{ss:Shim_n}.

\subsection{Formulation of the main theorems}  \label{ss:main theorems}
\begin{thm}  \label{t:1}
There exists a smooth algebraic stack $\BT_{1,\BF_p}^G$ over $\BF_p$ whose restriction to $\Sm_{\BF_p}^{\op}$ is the functor \eqref{e:BT_1^G on Sm}.
\end{thm}

By Lemma~\ref{l:easy lemma}, $\BT_{1,\BF_p}^G$ is unique.
Combining \S\ref{sss:forgetting nabla} and Lemma~\ref{l:easy lemma}, we get a morphism 
\begin{equation} \label{e:BT_1^G to Disp_1^G}
\BT_{1,\BF_p}^G\to\Disp_1^G
\end{equation} 

\begin{thm}  \label{t:2}
The morphism \eqref{e:BT_1^G to Disp_1^G} is an fppf gerbe banded by the group scheme $\Lau_1^G$ from \S\ref{sss:Lau_1}.
\end{thm}

Theorems~\ref{t:1} and \ref{t:2} will be proved in \S\ref{ss:proof of main theorems}.

\begin{cor}    \label{c:smooth algebraic}
$\BT_{1,\BF_p}^G$ is a quasi-compact smooth algebraic stack over $\BF_p$ of pure dimension~$0$ with affine diagonal. 
\end{cor}

\begin{proof}  
By Theorems~\ref{t:1}-\ref{t:2}, this follows from similar properties of $\Disp_1^G$ (see \S\ref{sss:Disp_1^G as a quotient}).
\end{proof}

\subsubsection{Remark}  \label{sss:BT_1^G smooth over Disp_1^G}
As explained in \S\ref{sss:gerbes are smooth} of Appendix~\ref{s:gerbe generalities}, Theorem~\ref{t:2} implies that the morphism \eqref{e:BT_1^G to Disp_1^G} is smooth.

\subsubsection{A simple example}  \label{sss:G=G_a}
Let $G$ be $\BG_a$ equipped with the usual action of $\BG_m$. In this case $\Disp_1^G=\Spec\BF_p$ and $\Lau_1^G=\mu_p$. Since $H^i_{\fppf} (\Spec\BF_p,\mu_p)=0$ for all $i$, Theorem~\ref{t:2} implies that $\BT_{1,\BF_p}^G$ is canonically isomorphic to the classifying stack of $\mu_p$ over $\BF_p$.

\subsection{A conjecture about $\BT_{1,\BF_p}^G$ in the case $G=GL(d)$}  \label{ss:conjecture for G=GL(d)}
\subsubsection{Some results of E.~Lau}
Let $d,d'$ be integers such that $0\le d'\le d$. Let $n\in\BN$. Let $G$ be the group scheme $GL(d)$ equipped with the $\BG_m$-action corresponding to the composite map 
\begin{equation}    \label{e:G_m to PGL(d)}
\BG_m\to GL(d)\to PGL(d),
\end{equation}
where the first map takes $t$ to a diagonal matrix with $d'$ diagonal entries equal to $t$ and $d-d'$ diagonal entries equal to $1$.

Let $\BT_n^{d,d'}$ be the stack of $n$-truncated Barsotti-Tate groups of height $d$ and dimension $d'$.
Using a covariant version of Dieudonn\'e theory\footnote{Those who prefer \emph{contra}variant Dieudonn\'e theory should replace $d'$ by $d-d'$ in the definition of \eqref{e:G_m to PGL(d)}.},  E.~Lau defined in \cite{Lau13} a canonical morphism $\BT_n^{d,d'}\otimes\BF_p\to\Disp_n^G$. Moreover, according to Theorem B of \cite{Lau13}, this morphism is a gerbe banded by a commutative locally free finite group scheme over $\Disp_n^G$. We denote this group scheme by $\Lau_n^{G, {\rm true}}$. 

According to \cite{Lau13}, the group scheme $\Lau_n^{G, {\rm true}}$ is infinitesimal and has order $p^{nd'(d-d')}$ (see Theorem B of \cite{Lau13} and the paragraph after it).
Moreover, Remark~4.8 of \cite{Lau13} describes the restriction of $\Lau_n^{G, {\rm true}}$ to the generic locus of $\Disp_n^G$; in the case $n=1$, it is canonically isomorphic to the restriction of $\Lau_1^G$ (which was described in \S\ref{sss:Lau_1^G on generic locus}).

\begin{prop}   \label{conj:1}
The isomorphism between the restrictions of $\Lau_1^{G, {\rm true}}$ and $\Lau_1^G$ to the generic locus of $\Disp_1^G$ extends\footnote{Such an extension is unique. Indeed, by \S\ref{sss:Disp_1^G as a quotient}, $\Disp_1^G$ is a quotient of the reduced irreducible scheme $G=GL(n)$.} to an isomorphism over the whole $\Disp_1^G$.
\end{prop}

\begin{proof}
Follows from \cite[Thm.~4.4.2(ii)]{Dr23} and \cite[\S~9.2.2]{Dr23}.
\end{proof}

Combining Proposition~\ref{conj:1} with Theorem~\ref{t:1}, we see that $\BT_{1,\BF_p}^G$ and $\BT_1^{d,d'}\otimes\BF_p$ are gerbes over $\Disp_1^G$ banded by the same group scheme $\Lau_1^G$.

\begin{conj}   \label{conj:2}
These two gerbes are isomorphic.
\end{conj}

As far as I understand, Conjecture~\ref{conj:2} has already been proved in \cite{GMM}.

\section{Analyzing the conditions for the connection $\nabla$}  \label{s:the conditions for nabla}
The definition of $\BT_1^G(S)$ from \S\ref{sss:the key definition} involves a connection $\nabla$, which has to satisfy certain conditions.
In this section we show that these conditions are affine-$\BF_p$-linear; for a precise statement, see Theorem~\ref{t:torsor} below. In \S\ref{s:proof of main theorems} we will deduce Theorems~\ref{t:1}-\ref{t:2}
from Theorem~\ref{t:torsor}.

\subsection{The sheaves $\cT$ and $\cT'$}
\subsubsection{The category $\Sm/\Disp_1^G$}
 Let $\Sm/\Disp_1^G$ be the category of pairs $(S,f)$, where $S\in\Sm_{\BF_p}$ and $f\in\Disp_1^G(S)$. We equip $\Sm/\Disp_1^G$ with the etale topology.

\subsubsection{The sheaf $\cT$}  \label{sss:the sheaf cT}
Let $(S,f)\in\Sm/\Disp_1^G$, so $f\in\Disp_1^G(S)$ is given by data (i)-(iii) from \S\ref{sss:the key definition}. We will use the notation of \S\ref{sss:the key definition} for these data (i.e., $\cF^\pm$, $\cF^\pm_M$, etc.). Let $\cT (S,f)$ be the fiber of the functor $\BT_1^G(S)\to\Disp_1^G(S)$ over $f\in\Disp_1^G(S)$. This fiber is a \emph{set} rather than a groupoid; namely, $\cT (S,f)$ is the set of connections $\nabla$ on the $P^+$-bundle $\cF^+$ satisfying certain conditions. The conditions are as follows:

(a) the connection on $\cF^+_M$ induced by $\nabla$ is equal to the one that comes from the isomorphism $\cF^+_M\iso\Fr_S^*\cF^-_M$;

(b) $\nabla$ is integrable;

(c) $\nabla$ satisfies the Katz condition, i.e.,
  \begin{equation}   \label{e:3Katz condition}
\pCurv_\nabla=-\KS_\nabla ,   
 \end{equation}
 where $\KS_\nabla\in H^0(S,(\fg_1)_{\cF^-}\otimes\Omega^1_S)$  is the Kodaira-Spencer 1-form defined in \S\ref{sss:defining KS} and 
 $\pCurv_\nabla\in H^0(S, (\Fr_*(\fg_1)_{\cF^+})^\nabla\otimes\Omega^1_S)=H^0(S, (\fg_1)_{\cF^-}\otimes\Omega^1_S)$ is the $p$-curvature of $\nabla$ (see \S\ref{sss:the key definition} for details).
 
 The assignment $(S,f)\mapsto \cT (S,f)$ is a sheaf on $\Sm/\Disp_1^G$ (with respect to the etale topology).

 \subsubsection{The sheaf $\cT'\supset\cT$}
 We keep the notation of \S\ref{sss:the sheaf cT}. Let $\cT' (S,f)$ be the set of connections $\nabla$ on $\cF^+$ satisfying condition (a) from \S\ref{sss:the sheaf cT}.
 Then $\cT'$ is a sheaf of sets on $\Sm/\Disp_1^G$, and $\cT\subset\cT'$.

 \subsection{$\cT$ and $\cT'$ as torsors}  \label{ss:cA and cA'}
 \subsubsection{The sheaf $\cA'$}   \label{sss:cA'}
 Define a sheaf of $\BF_p$-vector spaces $\cA'$ on $\Sm/\Disp_1^G$ as follows: 
 \begin{equation}  \label{e:cA'}
\cA' (S,f):=H^0(S,(\fg_1)_{\cF^+}\otimes\Omega^1_S).
 \end{equation}
 Then $\cA'$ acts on $\cT'$ in the usual way (adding a 1-form to a connection). Moreover, $\cT'$ is an $\cA'$-torsor.
 
 Using the isomorphism $\cF^+\iso\Fr_S^*\cF^-$ (which is a part of the data), we can rewrite \eqref{e:cA'} as
  \begin{equation}  \label{e:rewriting cA'}
\cA' (S,f)=H^0(S,\Fr_S^*(\fg_1)_{\cF^-}\otimes\Omega^1_S)=H^0(S,(\fg_1)_{\cF^-}\otimes(\Fr_S)_*\Omega^1_S).
 \end{equation}
 Note that $\Fr_S :S\to S$ induces the identity on the underlying set of $S$, so $(\Fr_S)_*\Omega^1_S$ is just the sheaf $\Omega^1_S$ equipped with the Frobenius-twisted $\cO_S$-action.

  \subsubsection{The next goals}
  In \S\ref{sss:cA} we will define a subsheaf of $\BF_p$-vector spaces $\cA\subset\cA'$. Then we will formulate Theorem~\ref{t:torsor}, which says that $\cT$ is an $\cA$-torsor.

  \subsubsection{The maps $\varphi$, $\tilde\varphi$, $C$, and $\tilde C$}  \label{sss:varphi and C}
  Let $\varphi :(\fg_1)_{\cF^-}\to (\fg_1)_{\cF^-}$ be the $p$-linear map corresponidng to the $\cO_S$-linear map 
  $\Fr_S^*(\fg_1)_{\cF^-}\to (\fg_1)_{\cF^-}$ from formula \eqref{e:p-operation2}. The latter is the composition
  \[
  \Fr_S^*(\fg_1)_{\cF^-}\iso (\fg_1)_{\cF^+}\mono \fg_{\cF^+}\iso\fg_{\cF^-}\epi (\fg_1)_{\cF^-}\, .
  \]
  
 The map $\varphi$ induces a $p$-linear map
  \[
  \tilde\varphi :(\fg_1)_{\cF^-}\otimes (\Fr_S)_*\Omega^1_S\to (\fg_1)_{\cF^-}\otimes\Omega^1_S\, ;
  \]
namely, $\tilde\varphi$ is the tensor product of $\varphi :(\fg_1)_{\cF^-}\to (\fg_1)_{\cF^-}$ and the identity\footnote{See the end of \S\ref{sss:cA'}.} map 
$$(\Fr_S)_*\Omega^1_S\to \Omega^1_S$$
(the latter is $p$-linear).

We have the $\cO_S$-submodule $((\Fr_S)_*\Omega^1_S)_{\closed}\subset (\Fr_S)_*\Omega^1_S$ and the Cartier operator\footnote{In Lemma~\ref{l:coordinate description of C} we will recall the explicit description of the Cartier operator assuming that $S$ is equipped with a coordinate system.}
\[
C:((\Fr_S)_*\Omega^1_S)_{\closed}\to \Omega^1_S,
\]
which is  $\cO_S$-linear and surjective. $C$ induces a surjective $\cO_S$-linear map
 \[
\tilde C:(\fg_1)_{\cF^-}\otimes ((\Fr_S)_*\Omega^1_S)_{\closed}\to (\fg_1)_{\cF^-}\otimes  \Omega^1_S, \quad \tilde C:=\id\otimes C
\]
\subsubsection{The sheaf $\cA$}  \label{sss:cA}
Let us use formula~\eqref{e:rewriting cA'} for $\cA'$. Define a sheaf of $\BF_p$-vector spaces $\cA\subset\cA'$ as follows:
\begin{equation}  \label{e:cA}
\cA(S,f)=\{\omega\in H^0(S,(\fg_1)_{\cF^-}\otimes ((\Fr_S)_*\Omega^1_S)_{\closed}\,|\,\tilde C(\omega )=\tilde\varphi (\omega)\},
\end{equation}
where  $\tilde C$ and $\tilde\varphi$ were defined in \S\ref{sss:varphi and C}.

\subsubsection{Remarks about $\cA$} \label{sss:remarks about cA}

(i) By definition, $\cA$ is the kernel of a certain morphism of sheaves. Later we will see that this morphism is surjective, see Lemma~\ref{l:surjectivity of C-varphi}.

(ii) The ``true nature'' of $\cA$ will be explained in Proposition~\ref{p:cB=cA}.
  
\begin{thm} \label{t:torsor}
The subsheaf $\cT\subset\cT'$ is stable under the action of $\cA\subset\cA'$. Moreover, $\cT$ is an $\cA$-torsor.
\end{thm} 

The proof will be given in \S\ref{ss:proof of the torsor theorem}-\ref{ss:change of p-Curv}.

\begin{rem}
In particular, Theorem~\ref{t:torsor} says that for every $S\in\Sm_{\BF_p}$, every object of $\Disp_1^G(S)$ admits a lift to $\BT_1^G(S)$ etale-locally on $S$.
As explained in \S\ref{sss:BT_1^G smooth over Disp_1^G}, this is a part of Theorem~\ref{t:2}, which we want to prove.
\end{rem}

\subsection{Proof of Theorem~\ref{t:torsor}}  \label{ss:proof of the torsor theorem}
Given $(S,f)\in\Sm/\Disp_1^G$, let $\cT_S$ be the restriction of $\cT$ to~$S_{\et}$; similarly, we have $\cT'_S$, $\cA_S$, $\cA'_S$. The problem is to prove that
the subsheaf $\cT_S\subset\cT'_S$ is an $\cA_S$-torsor.

\begin{lem}  \label{l:nabla_0}
Zariski-locally on $S$, there exists a connection $\nabla$ on $\cF^+$ satisfying conditions (a)-(b) from \S\ref{sss:the sheaf cT}.
\end{lem}

\begin{proof} 
Zariski-locally, there exists an isomorphism between $\cF^+$ and the $P^+$-torsor induced from the $M$-torsor $\cF_M^+$ via the inclusion $M\mono P^+$.
Choose such an isomorphism and take $\nabla$ to be induced by the canonical connection on $\cF^+_M=\Fr_S^*\cF_M^-$.
\end{proof} 

\begin{lem}  \label{l:nabla-nabla_0}
Let $\nabla$ be as in Lemma~\ref{l:nabla_0}. A connection $\tilde\nabla$ on $\cF^+$ satisfies conditions (a)-(b) from \S\ref{sss:the sheaf cT} if and only if
$\tilde\nabla=\nabla+\omega$, where $\omega\in H^0(S, (\fg_1)_{\cF^+}\otimes\Omega^1_S)$ and $d\omega=0$.
\end{lem}

\begin{proof} 
Write $\tilde\nabla=\nabla+\omega$, where $\omega\in H^0(S, (\fg_{\ge 0})_{\cF^+}\otimes\Omega^1_S)$. Condition (a) means that
$\omega\in H^0(S, (\fg_1)_{\cF^+}\otimes\Omega^1_S)$. Condition (b) is equivalent to the Maurer-Cartan equation for~$\omega$. 
Since $[\fg_1,\fg_1]=0$, this is just the equation $d\omega=0$.
\end{proof} 

As noted in \S\ref{sss:cA'}, $H^0(S,(\fg_1)_{\cF^+}\otimes\Omega^1_S)=H^0(S,(\fg_1)_{\cF^-}\otimes(\Fr_S)_*\Omega^1_S)$.
So the $\omega$ from Lemma~\ref{l:nabla-nabla_0} is in $H^0(S,(\fg_1)_{\cF^-}\otimes((\Fr_S)_*\Omega^1_S)_{\closed})$.

\begin{lem}  \label{l:change of KS & p-Curv}
In the situation of Lemma~\ref{l:nabla-nabla_0}, we have
\begin{equation}   \label{e:change of KS}
\KS_{\tilde\nabla} -\KS_{\nabla}=\tilde\varphi (\omega),
\end{equation}
\begin{equation}     \label{e:change of p-Curv}
\pCurv_{\tilde\nabla} -\pCurv_{\nabla}=-\tilde C (\omega),
\end{equation}
where $\tilde\varphi :(\fg_1)_{\cF^-}\otimes (\Fr_S)_*\Omega^1_S\to (\fg_1)_{\cF^-}\otimes\Omega^1_S$ and 
$\tilde C:(\fg_1)_{\cF^-}\otimes ((\Fr_S)_*\Omega^1_S)_{\closed}\to (\fg_1)_{\cF^-}\otimes  \Omega^1_S$ are as in \S\ref{sss:varphi and C}.
\end{lem}

Formula \eqref{e:change of KS} immediately follows from the definition of $\tilde\varphi$. A proof of \eqref{e:change of p-Curv} will be given in \S\ref{ss:change of p-Curv}.

Theorem~\ref{ss:proof of the torsor theorem} follows from \eqref{e:change of KS}-\eqref{e:change of p-Curv} and the next lemma.

\begin{lem}  \label{l:surjectivity of C-varphi}
The map $(\fg_1)_{\cF^-}\otimes ((\Fr_S)_*\Omega^1_S)_{\closed}\overset{\tilde C-\tilde\varphi}\longrightarrow (\fg_1)_{\cF^-}\otimes  \Omega^1_S$ is a surjective morphism of etale sheaves of 
$\BF_p$-vector spaces.
\end{lem}

\begin{proof} 
Let $\cE_1:=(\fg_1)_{\cF^-}\otimes ((\Fr_S)_*\Omega^1_S)_{\closed}\,$, $\cE_2:=(\fg_1)_{\cF^-}\otimes  \Omega^1_S$; these are finitely generated locally free $\cO_S$-modules.  Let $E_1,E_2$ be the corresponding vector bundles; these are schemes over $S$ (namely, $E_i$ is the spectrum of the symmetric algebra of $\cE_i^*$). Recall that $\tilde C:\cE_1\to\cE_2$ is a surjective $\cO_S$-linear map and $\tilde\varphi:\cE_1\to\cE_2$ is $p$-linear. So the morphism of $S$-schemes $E_1\overset{\tilde C-\tilde\varphi}\longrightarrow E_2$ is smooth. It is also surjective: indeed, the image is an open subgroup scheme of~$E_2$, so it has to be equal to $E_2$. Therefore every section of $E_2$ etale-locally admits a lift to a section of $E_1$.
\end{proof}

\subsection{Proof of formula~\eqref{e:change of p-Curv}}      \label{ss:change of p-Curv}
A proof will be given in \S\ref{sss:change of p-Curv}; it is parallel to that of Propostion~7.1.2 of \cite{K72}. Let us first recall some facts used in the proof.

\subsubsection{Recollections} 
Let $\BF_p\langle x,y\rangle$ be the free associative $\BF_p$-algebra on $x,y$. Let us define $J\in\BF_p\langle x,y\rangle$ by $J(x,y):=(x+y)^p-x^p-y^p$. 
We have $J=J_1+\ldots +J_{p-1}$, where $J_i$ is homogeneous of degree $i$ in $y$. Jacobson proved that each $J_i$ belongs to the free Lie algebra on $x,y$ (which is a Lie subalgebra of $\BF_p\langle x,y\rangle$) and that

\begin{equation}   \label{e:J_1}
J_1(x,y)=\ad_x^{p-1}(y).
\end{equation}
Hochschild's proof of these results can be found in \cite[p.~199-200]{C58}. Since $J_i$ is a Lie polynomial in $x,y$, which is homogeneous of degree $i$ in $y$, we see that
\begin{equation}   \label{e:J_i for i>1}
J_i\in [\fa ,\fa ] \quad \mbox{ for } i>1,
\end{equation}
where $\fa\subset\BF_p\langle x,y\rangle$ is the Lie subalgebra generated by the elements $\ad_x^j(y)$, $j\ge 0$.

\medskip

We will also need the following coordinate description of the Cartier operator
\[
C:H^0(S,\Omega^1_S)_{\closed}\to H^0(S,\Omega^1_S).
\]

\begin{lem}   \label{l:coordinate description of C}
Let $S$ be a scheme etale over $\Spec \BF_p[x_1,\ldots ,x_n]$ and  
$\omega\in H^0(S,\Omega^1_S)_{\closed}\,$. Write
$\omega=\sum\limits_i f_i\cdot dx_i$, $C(\omega )=\sum\limits_i h_i \cdot dx_i$, where $f_i,h_i\in H^0(S,\cO_S)$. 
 Then
\begin{equation}   \label{e:coordinate description of C}
h_i^p=-\partial_i^{p-1}(f_i), \quad \mbox{where } \partial_i:=\frac{\partial}{\partial x_i}\, .
\end{equation}
\end{lem}

This lemma is well known, see formula~(7.1.2.6) of \cite{K72}; moreover, P.~Cartier used \eqref{e:coordinate description of C} to define $C$, see p.~200 and p.~202-203 of \cite{C58}. We prove the lemma for completeness.

\begin{proof}
By Cartier's isomorphism, the space of closed 1-forms 
is generated by locally exact 1-forms and by 
\begin{equation}  \label{e:simple closed 1-form}
g^p\cdot x_i^{p-1}dx_i, \quad g\in H^0(S,\cO_S), 1\le i\le n.
\end{equation}
If $\omega$ is given by \eqref{e:simple closed 1-form} then $C(\omega )=g\cdot dx_i$, and \eqref{e:coordinate description of C} holds because 
\[
\partial_i^{p-1}(g^p\cdot x_i^{p-1})=(p-1)!\cdot g^p=-g^p
\]
 by Wilson's theorem.
On the other hand, if $\omega=\sum\limits_i f_i\cdot dx_i$ is exact then $C(\omega )=0$, and the problem is to show that $\partial_i^{p-1}(f_i)=0$. Indeed, if $\omega=du$ then $f_i=\partial_i (u)$ and 
we have $\partial_i^{p-1}(f_i)=\partial_i^{p}(u)=0$.
\end{proof}

\begin{rem}
Formula~\ref{e:simple closed 1-form} tells us that in the situation of Lemma~\ref{l:coordinate description of C}, the function $\partial_i^{p-1}(f_i)$ is a $p$-th power. Here is a direct proof of this. It suffices to show that $\partial_j\partial_i^{p-1}(f_i)=0$ for all~$j$. But the equality $d\omega=0$ means that $\partial_j(f_i)=\partial_i(f_j)$, so $\partial_j\partial_i^{p-1}(f_i)=\partial_i^p(f_j)=0$.
\end{rem}

\subsubsection{Proof of formula~\eqref{e:change of p-Curv}}      \label{sss:change of p-Curv}
We can assume that $S$ is equipped with an etale morphism to $\Spec \BF_p[x_1,\ldots ,x_n]$. 
Write $\omega\in H^0(S, (\fg_1)_{\cF^+}\otimes\Omega^1_S)_{\closed}$ as
$$\omega =\sum\limits_i f_i\otimes dx_i\, , \mbox{ where } f_i\in H^0(S,(\fg_1)_{\cF^+}).$$ 
We have $(\fg_1)_{\cF^+}=\Fr^*\cE$, where $\cE=(\fg_1)_{\cF^-}$; so $\cE$ is the horizontal part of $\Fr_*(\fg_1)_{\cF^+}$. The problem is
to prove that
\begin{equation}  \label{e:2change of p-Curv}
\pCurv_{\nabla +\omega} -\pCurv_{\nabla}=-\tilde C(\omega ),
\end{equation}
where
$\tilde C:H^0(S, (\fg_1)_{\cF^+}\otimes\Omega^1_S)_{\closed}=H^0(S,\cE\otimes (\Fr_*\Omega^1_S)_{\closed})\to
H^0(S,\cE\otimes \Omega^1_S)$
is induced by $C:(\Fr_*\Omega^1_S)_{\closed}\to\Omega^1_S$. 
Write $\tilde C(\omega )=\sum\limits_i h_i\otimes dx_i$, where   $h_i\in H^0(S,\cE )$, then
Lemma~\ref{l:coordinate description of C} implies that 
\begin{equation} \label{e:coordinate description of tilde C}
\Fr^*(h_i)=-\nabla_i^{p-1}(f_i), \quad \mbox{where }\nabla_i:=\nabla_{\!\!\frac{\partial}{\partial x_i}}.
\end{equation}

Let us now compute the l.h.s. of \eqref{e:2change of p-Curv}.
By \S\ref{sss:U^+=g_1}, $[\fg_1,\fg_1]=0$ and the $p$-operation on $\fg_1$ is zero. So by \eqref{e:J_1}-\eqref{e:J_i for i>1}, we get
$(\nabla_i+f_i)^p-\nabla_i^p=\nabla_i^{p-1}(f_i)$. This means that the l.h.s. of \eqref{e:2change of p-Curv} equals $-\sum\limits_i h_i\otimes dx_i$, where
the $h_i$'s are as in \eqref{e:coordinate description of tilde C}. \qed

\section{Proof of Theorems~\ref{t:1}-\ref{t:2}}   \label{s:proof of main theorems}
In \S\ref{ss:proof of main theorems} we will see that Theorems~\ref{t:1}-\ref{t:2} easily follow from Theorem~\ref{t:torsor} and a result of Artin-Milne \cite{AM}, which we are going to recall now.

\subsection{A result of Artin-Milne}  \label{ss:Artin-Milne}
\subsubsection{The setting}   \label{sss:Artin-Milne setting}
Let $S$ be  a smooth $\BF_p$-scheme and $H$ a finite flat commutative group scheme over $S$. Assume that $H$ has height 1 (i.e., is killed by Frobenius). We have a morphism of sites
$\pi :S_{\fppf}\to S_{\et}$. Artin and Milne described the sheaves $R^q\pi_*H$.

To formulate their result, we need some notation. Let $\fh:=\Lie (H)$; this is a finitely generated locally free $\cO_S$-module. The $p$-operation on $\fh$ is a $p$-linear map $\varphi :\fh\to\fh$.
Similarly to \S\ref{sss:varphi and C}, one defines a $p$-linear map
$ \tilde\varphi :\fh\otimes (\Fr_S)_*\Omega^1_S\to \fh\otimes\Omega^1_S$
and  a surjective $\cO_S$-linear map
$\tilde C :\fh\otimes ((\Fr_S)_*\Omega^1_S)_{\closed}\to\fh\otimes  \Omega^1_S$
(the definition of $\tilde\varphi$ uses $\varphi$, and the definition of $\tilde C$ uses the Cartier operator $C$).

\begin{prop} \label{p:Artin-Milne}
In the situation of \S\ref{sss:Artin-Milne setting} one has a canonical exact sequence 
\begin{equation}   \label{e:Artin-Milne}
0\to R^1\pi_*H\overset{f}\longrightarrow\fh\otimes ((\Fr_S)_*\Omega^1_S)_{\closed}\overset{\tilde C-\tilde\varphi}\longrightarrow\fh\otimes\Omega^1_S\to 0
\end{equation}
of sheaves on $S_{\et}$, which is functorial with respect to $H$ and with respect to base changes $S'\to S$; moreover, $R^q\pi_*H=0$ if $q>1$.
\end{prop}

This is Proposition~2.4 of  \cite{AM}. The construction of the map $f$ from \eqref{e:Artin-Milne} will be recalled in \S\ref{sss:Artin-Milne interpreted and generalized}.

\subsubsection{Example}
Let $H=(\alpha_p)_S$. In this case  Propostion~\ref{p:Artin-Milne} says that $$R\pi_*H=((\Fr_S)_*\Omega^1_S)_{\exact}[-1].$$ 
Since $\Fr_S$ induces the identity functor $S_{\et}\to S_{\et}$, we can rewrite this as $R\pi_*H=\Omega^1_{S,\exact}[-1]$. This is well known (and explained at the end of \S\ref{sss:alpha_p-torsors}).

\subsubsection{Remarks}
(i) The equality $R^0\pi_*H=0$ is clear because $S$ is reduced and the reduced part of the fiber of $H$ over each point of $S$ is zero.

(ii) The proof of the equality $R^q\pi_*H=0$ for $q>1$ given in \cite{AM} works for \emph{any} scheme $S$ and \emph{any} finite locally free commutative group scheme over $S$.

(iii) In a particular situation, surjectivity of the map $\tilde C-\tilde\varphi$ from \eqref{e:Artin-Milne} was proved in Lemma~\ref{l:surjectivity of C-varphi}. The same argument works in general.

\subsubsection{Interpretation and noncommutative generalization of \eqref{e:Artin-Milne}}  \label{sss:Artin-Milne interpreted and generalized}
An element of $H^1_{\fppf}(S,H)$ is an isomorphism class of a principal $H$-bundle $E\to S$. Since $H$ is killed by Frobenius, the geometric Frobenius $F:E\to\Fr_S^*E$ factors as $E\to S\overset{\sigma}\longrightarrow\Fr_S^*E$. The section $\sigma$ trivializes the bundle $\Fr_S^*E$. On the other hand, any $\Fr_S$-pullback (e.g., $\Fr_S^*E$) is equipped with a $p$-integrable connection. A connection 
$\nabla$ on the trivialized $\Fr_S^*H$-bundle $\Fr_S^*E$ is given by an element $\omega\in H^0(S, \Fr_S^*\fh\otimes\Omega^1_S)=H^0(S, \fh\otimes(\Fr_S)_*\Omega^1_S)$. Moreover, integrability of 
$\nabla$ is equivalent to the Maurer-Cartan equation for $\omega$; since $H$ is commutative, this equation just means that 
$\omega\in H^0(S, \fh\otimes ((\Fr_S)_*\Omega^1_S))_{\closed})$.
Thus we get a map 
$$H^1_{\fppf}(S,H)\to H^0(S, \fh\otimes ((\Fr_S)_*\Omega^1_S))_{\closed}).$$ 
This map and similar maps for all schemes etale over $S$ give rise to the map $f$ from \eqref{e:Artin-Milne}. (This is essentially the description of $f$ from \cite[\S 2.7]{AM} although the word ``connection'' is not used there).

It is easy to check\footnote{The verification is parallel to \S\ref{sss:change of p-Curv}; the only difference is that the $p$-operation on $\fh$ is not assumed to be zero, while the restricted Lie algebra $(\fg_1)_{\cF^+}$ from \S\ref{sss:change of p-Curv} has zero $p$-operation.} that the inclusion $\im f\subset\Ker (\tilde C-\tilde\varphi )$ (which is a part of Proposition~\ref{p:Artin-Milne}) just means that the $p$-curvature of $\nabla$ is zero. So the exact sequence \eqref{e:Artin-Milne} essentially says that \emph{$H^1_{\fppf}(S,H)$ canonically identifies (via $\Fr_S$-pullback) with the set of $p$-integrable connections on the trivial $H$-bundle on $S$.} One can prove that this statement remains valid without assuming $H$ commutative; we will not need this fact.

\subsection{The classifying stack of $\Lau_1^G$}   \label{ss:BLau_1}
\subsubsection{The sheaf $\cB$}   \label{sss:cB}
Let $\Sch/\Disp_1^G$ be the category of pairs $(S,f)$, where $S$ is a scheme and $f:S\to\Disp_1^G$ is a morphism. Let $\Sm/\Disp_1^G$ be the full subcategory of $\Sch/\Disp_1^G$ formed by pairs 
$(S,f)$ such that $S\in\Sm_{\BF_p}$. We equip $\Sch/\Disp_1^G$ and $\Sm/\Disp_1^G$ with the etale topology.

In \S\ref{sss:Lau_1} we defined a commutative finite flat group scheme $\Lau_1^G$ over $\Disp_1^G$; this group scheme has height 1.
By definition, the classifying stack of $\Lau_1^G$ is the stack of Picard groupoids on $\Sch/\Disp_1^G$ whose sections over $(S,f)\in\Sch/\Disp_1^G$ are $(f^*\Lau_1^G)$-torsors.
Let $\cB$ be the restriction of this stack to $\Sm/\Disp_1^G$. For any $(S,f)\in\Sm/\Disp_1^G$ one has $H^0(S,f^*\Lau_1^G)=0$ (because $S$ is reduced), so the groupoid of sections of $\cB$ over $(S,f)$ is discrete. Therefore $\cB$ is just a sheaf of abelian groups.

\subsubsection{Remark}
Restricting from $\Sch/\Disp_1^G$ to $\Sm/\Disp_1^G$ does not lead to loss of information: this follows from Lemma~\ref{l:easy lemma} and the fact that the classifying stack of $\Lau_1^G$ is smooth over 
$\BF_p$. The latter follows from \S\ref{sss:Classifying stacks} of Appendix~\ref{s:gerbe generalities} because by \S\ref{sss:Disp_1^G as a quotient}, $\Disp_1^G$ is smooth 
over~$\BF_p$.

\begin{prop}    \label{p:cB=cA}
The sheaf $\cB$ from \S\ref{sss:cB} is canonically isomorphic to the sheaf $\cA$ from~\S\ref{sss:cA}.
\end{prop}

\begin{proof}
Follows from Proposition~\ref{p:Artin-Milne} and the definitions of $\cA$ and $\Lau_1^G$.
\end{proof}

\subsection{Proof of Theorems~\ref{t:1}-\ref{t:2}}   \label{ss:proof of main theorems}
In \S\ref{sss:the sheaf cT} we introduced the notation $\cT$ for $\BT_{1,\BF_p}^G$ viewed as a sheaf of sets on $\Sm/\Disp_1^G$. The problem is to show that $\cT$ is locally isomorphic to the sheaf $\cB$ from \S\ref{sss:cB}. By Proposition~\ref{p:cB=cA}, $\cB=\cA$. By Theorem~\ref{t:torsor}, $\cT$ is locally isomorphic to $\cA$. \qed

\appendix

\section{Generalities on stacks} \label{s:stacky nonsense}
Let $P$ be a smooth-local property of schemes. Let $\Sch$ be the category of all schemes and $\Sch_P$ the category of schemes satisfying $P$; we equip $\Sch$ and $\Sch_P$ with the etale topology.
Let $\Stacks_P$ be the (2,1)-category of algebraic stacks satisfying $P$ (this makes sense because $P$ is smooth-local).
Let $\Grpds$ be the (2,1)-category of groupoids.

Given a stack $\sY:\Sch^{\op}\to\Grpds$, let $\sY_P$ denote the restriction of $\sY$ to $\Sch_P^{\op}$.

\begin{prop}   \label{p:stacky Yoneda}
Let $\sX\in\Stacks_P$. Then the functor $\Mor (\sX,\sY )\to\Mor (\sX_P,\sY_P )$ is an equivalence for any stack $\sY :\Sch^{\op}\to\Grpds$.
\end{prop}

\begin{proof}
Let $\pi :X\to\sX$ be a smooth surjective morphism with $X$ being a scheme. Let $X_{\bullet}$ be the  \v {C}ech nerve of $\pi$. Note that each $X_n$ is an algebraic space with property $P$.
We have
\[
\Mor (\sX,\sY )=\Tot(\Mor (X_{\bullet},\sY )), \quad  \Mor (\sX_P,\sY_P )=\Tot(\Mor ((X_{\bullet})_P,\sY_P )).
\]
Thus we have reduced the problem to the case where $\sX$ is an algebraic space.

Running the same argument again, we reduce to the case where $\sX\in\Sch_P$, which is covered by Yoneda's lemma.
\end{proof}

Our next goal is to reformulate Proposition~\ref{p:stacky Yoneda} in terms of a certain morphism of toposes.
Our approach is influenced by \S 4.10 of Expos\'e IV of \cite{SGA4} and by the more sophisticated Proposition~A.0.4 of \cite{EHKSY} (which goes back to A.~Mathew and is about \emph{derived} stacks).

Let $\widetilde{\Sch}$ (resp.~$\widehat{\Sch}$) be the category of presheaves (resp.~sheaves) of groupoids on $\Sch$. Similaly, we have
$\widetilde{\Sch_P}$ and $\widehat{\Sch_P}$. Let $\tilde g_*:\widetilde{\Sch}\to\widetilde{\Sch_P}$, $\hat g_*:\widehat{\Sch}\to\widehat{\Sch_P}$ be the restriction functors.
They have left adjoints $\tilde g^*$ and $\hat g^*$. The functor $\tilde g^*:\widetilde{\Sch_P}\to\widetilde{\Sch}$ is the \emph{left Kan extension}.
The functor $\hat g^*:\widehat{\Sch_P}\to\widehat{\Sch}$ is the \emph{sheafified left Kan extension}, i.e., the composition
\[
\widehat{\Sch_P}\mono\widetilde{\Sch_P}\overset{\tilde g^*}\longrightarrow\widetilde{\Sch}\to\widehat{\Sch}.
\]
Each of the adjoint pairs $(\tilde g^*,\tilde g_*)$ and $(\hat g^*,\hat g_*)$ defines a morphism of toposes.

Since the functor $\Sch_P\to\Sch$ is fully faithful, so is $\tilde g^*$; equivalently, the unit of the adjunction $\id\to\tilde g_*\tilde g^*$ is an isomorphism. 
The same is true for $\hat g^*$. For sheaves of sets, this is \cite[Tag 00XT]{Sta}. For sheaves of groupoids, one can argue as follows: if $\cF\in\widehat{\Sch_P}$ then $\hat g_*\hat g^*(\cF)$ is the restriction of the sheafification of $\tilde g^*\cF$, which equals\footnote{Sheafification commutes with restriction by part 1 of \cite[Prop.~7.1]{CM}. Probably this can also be proved by interpreting $\widehat{\Sch}$ and $\widehat{\Sch_P}$ as explained in part 2 of Exercise 4.10.6 of Expos\'e IV of \cite{SGA4} (i.e., a sheaf on $\Sch$ is just a collection of sheaves $\cF_S$ on $S_{\et}$ for all $S\in\Sch$ plus certain morphisms relating the sheaves $\cF_S$ with each other).} the sheafification of the restriction of $\tilde g^*\cF$, i.e., the sheafification of $\cF$, i.e., $\cF$ itself.

Now we can reformulate Proposition~\ref{p:stacky Yoneda} as follows.

\begin{prop}   \label{p:sheafified Kan extension}
If $\sX\in\Stacks_P$ then the canonical morphism 
$$\hat g^*\sX_P=\hat g^*\hat g_*\sX\to\sX$$
 is an isomorphism. Equivalently, $\sX$ belongs to the essential image of the fully faithful functor~$\hat g^*$. \qed
\end{prop}

Proposition~\ref{p:sheafified Kan extension} tells us how to reconstruct $\sX\in\Stacks_P$ from $\sX_P$: namely, $\sX=\hat g_*\sX_P$, i.e., 
$\sX$ is the sheafified left Kan extension of $\sX_P$.

\section{Recollections on gerbes}   \label{s:gerbe generalities}
\subsubsection{Classifying stacks}   \label{sss:Classifying stacks} 
Let $H$ be a flat affine group scheme of finite presentation over an algebraic stack $\sX$. Then the classifying stack $BH$ is known to be an algebraic stack smooth over $\sX$:
indeed, we can assume that $\sX$ is a scheme, in which case the statement is proved in \cite[Tag 0DLS]{Sta} using \cite[Tag 05B5]{Sta} and  a deep theorem of M.~Artin, see
\cite[Tag 06FI]{Sta}. (On the other hand, if $H$ is finite over $\sX$, one can give an elementary constructive proof similar to the one given in \S\ref{sss:alpha_p-torsors} in the case $H=\alpha_p$.)

\subsubsection{Gerbes}   \label{sss:gerbes are smooth} 
Let $f:\sY\to\sX$ be a morphism of finite presentation between algebraic stacks. If $f$ is an fppf gerbe then $f$ is smooth.
Indeed, smoothness can be checked fppf-locally on $\sX$, so we can assume that the gerbe is trivial. Then we can apply \S\ref{sss:Classifying stacks}.

\subsubsection{Remark}    \label{sss:gerbes are etale-locally trivial} 
Let $f:\sY\to\sX$ be as in \S\ref{sss:gerbes are smooth}. Since $f$ is smooth and surjective, it admits a section locally for the \emph{smooth} topology of $\sX$. 
This implies that if $\sX$ is a scheme then $f$ admits a section etale-locally on $\sX$.

\section{A definition of $\BT_n^G$ via syntomification}   \label{s:BT_n^G via syntomification}
Let $n\in\BN$ and $G\in\Shim_n\,$, where $\Shim_n$ is as in \S\ref{ss:Shim_n}. For every derived $p$-adic formal scheme $S$, we 
define in \S\ref{ss:BT_n^G via syntomification} an $\infty$-groupoid $\BT_n^G(S)$. The assignment $S\mapsto\BT_n^G(S)$ is an etale sheaf.  
Conjecture~\ref{conj:algebraicity} says that for each $m\in\BN$ the restriction of $\BT_n^G$ to the category of derived schemes over $\BZ/p^m\BZ$ is a smooth algebraic stack over $\BZ/p^m\BZ$.

\subsection{Recollections on syntomification}    \label{ss:Recollections on syntomification}
\subsubsection{The stack $S^\prism$}   \label{sss:S^prism}
Bhatt and Lurie \cite{BL2} define a \emph{prismatization} functor $S\mapsto S^\prism$ from the category of derived $p$-adic formal schemes to the category of fpqc-stacks of $\infty$-groupoids on the category of $p$-nilpotent\footnote{This means that $p$ is Zariski-locally nilpotent.} derived schemes. The stacks $S^\prism$ are not very far from being algebraic in the sense of Definition 2.3.5 of \cite{Prismatization} (whose essential point is that the quotient of a scheme over a ring $R$ by an action of a flat affine group scheme $H$  over $R$ is considered to be an algebraic stack even if $H$ has infinite type). 

The above words ``not very far'' are necessary for two reasons. First, the derived setting is not considered in \cite{Prismatization}. Second, already the stacks $(\Spf\BZ_p)^\prism$ and $(\Spec \BZ/p^n\BZ)^\prism$ are  \emph{formal} (in the sense of \cite[\S 2.9.1]{Prismatization}) rather than algebraic; e.g., $(\Spec \BF_p)^\prism=\Spf\BZ_p$.

\subsubsection{The stacks $S^\cN$, and $S^{\Syn}$}   \label{sss:S^Syn}
Bhatt defines in his lecture notes \cite{B} the \emph{(Nygaard-)filtered prismatization} functor $S\mapsto S^\cN$. As before, this is a functor from the category of derived $p$-adic formal schemes to the category of fpqc-stacks of $\infty$-groupoids on the category of $p$-nilpotent derived schemes. However, for pedagogical reasons, Bhatt assumes in \cite{B} that $S$ is a classical\footnote{``Classical'' means ``not really derived''.} scheme, and he defines in \cite{B} only the restriction of the stack $S^\cN$ to the category of classical $p$-nilpotent schemes. As before, the stacks $S^\cN$ are not very far from being algebraic in the sense of \cite{Prismatization}.

 The stack $S^\cN$ has two open substacks canonically isomorphic to $S^\prism$. Gluing them together, Bhatt gets a stack which is called the \emph{syntomification} of $S$ and denoted by $S^{\Syn}$. The assignment $S\mapsto S^{\Syn}$ is a functor. By definition,
 one has a natural morphism $$S^\cN\to S^{\Syn}.$$
 
There is a canonical line bundle on $(\Spf\BZ_p)^{\Syn}$ called the Breuil-Kisin twist. It defines a morphism $(\Spf\BZ_p)^{\Syn}\to B\BG_m$, where $B\BG_m$ is the classifying stack of $\BG_m$.
By functoriality, it induces a morphsim $S^{\Syn}\to B\BG_m$ for any $S$. The corresponding line bundle on $S^{\Syn}$ is denoted by $\cO_{S^{\Syn}}\{ 1\}$ or $\cO\{ 1\}$.

\subsubsection{The canonical morphism $S\times B\BG_m\to S^\cN$}   \label{sss:S times BG_m to S^N}
There is a canonical morphism 
\begin{equation}  \label{e:S times BG_m to  S^N}
S\times B\BG_m\to S^\cN
\end{equation}
over $B\BG_m$; namely, \eqref{e:S times BG_m to S^N} is the composite map
\[
S\times\{ 0\}/\BG_m\mono S\times\BA^1/\BG_m\to S^{dR,+}\to S^\cN,
\]
where $S^{dR,+}$ and the maps $S\times\BA^1/\BG_m\to S^{dR,+}\to S^\cN$ are defined in \cite[ \S 5.3.13]{B}.

We need the morphism \eqref{e:S times BG_m to S^N} only if $S=\Spec k$, where $k$ is a perfect field. In this case it can be described as follows. 
It is known\footnote{See \cite[\S 3.3 and  \S 5.4]{B}.} that 
\[
(\Spec k)^\cN\otimes\BF_p=(\Spec k[u,t]/(ut))/\BG_m\, ,
\]
 where the $\BG_m$-action is such that $\deg t=1$, $\deg u=-1$; so the closed substack of the stack
 $(\Spec k)^\cN\otimes\BF_p$ given by $t=u=0$ equals $\Spec k\times B\BG_m$. (Let us note that this closed substack is called the \emph{Hodge locus.})

\subsubsection{On notation}
The definition of $S^\cN$ and  $S^{\Syn}$ is sketched in \cite[\S 1.7]{Prismatization}, but the notation is different there ($S^\prismp$ instead of $S^\cN$ and $S^\prismpp$ instead of $S^{\Syn}$).
In \cite{BL,BL2} the stack $S^\prism$ was denoted by $\WCart_S\,$.

\subsection{Definition of $\BT_n^G(S)$}  \label{ss:BT_n^G via syntomification}
Let $S$ be a derived $p$-adic formal scheme. Let $G\in\Shim_n\,$. 

\subsubsection{The group scheme $G_{\cO\{ 1\}}$}  \label{sss:BK-twist of G}
Our $G$ is a group scheme over $\BZ/p^n\BZ$. Twisting $G$ by the canonical $\BG_m$-torsor on $B\BG_m\otimes \BZ/p^n\BZ $, one gets a group scheme over $B\BG_m\otimes \BZ/p^n\BZ $, which we 
denote by $G_{\cO\{ 1\}}$. One has
\[
\Lie (G_{\cO\{ 1\}})=\bigoplus_i\fg_i\otimes\cO\{ i\},
\]
where $\fg_i$ is the $i$-th graded component of $\Lie (G)$ and 
$\cO\{ i\}$ is the $i$-th tensor power of the canonical line bundle on $B\BG_m\otimes \BZ/p^n\BZ$.

Since $S^{\Syn}\otimes\BZ/p^n\BZ$ is equipped with a morphism to $B\BG_m\otimes \BZ/p^n\BZ $, we can consider $G_{\cO\{ 1\}}$-torsors on $S^{\Syn}\otimes\BZ/p^n\BZ$.

\subsubsection{Definition}   \label{sss:BT_n^G via syntomification}
For any derived $p$-adic formal scheme $S$ of finite type\footnote{``Finite type'' means that the classical truncation of $S\otimes\BF_p$ has finite type in the usual sense.
For the motivation of the finite type precaution, see \S\ref{sss:Good news}(ii) below.} over $\Spf\BZ_p$, let $\BT_n^G(S)$ be the $\infty$-groupoid of $G_{\cO\{ 1\}}$-torsors $\cE$ on $S^{\Syn}\otimes\BZ/p^n\BZ$ such that for every geometric point $\alpha :\Spec k\to S$, the pullback of $\cE$ via the composite morphism
\begin{equation}   \label{e:punctual triviality}
\Spec k\times B\BG_m\to (\Spec k )^\cN\otimes\BF_p\overset{\alpha_*}\longrightarrow S^\cN\otimes\BF_p\mono S^\cN\otimes\BZ/p^n\BZ\to S^{\Syn}\otimes\BZ/p^n\BZ
\end{equation}
is trivial. The first arrow in \eqref{e:punctual triviality} was described in \S\ref{sss:S times BG_m to S^N}.

\subsubsection{Remarks}   \label{sss:H^1(G_m, G)}  
(i) The isomorphism class of the above-mentioned pullback is an element $\nu_\alpha\in H^1((\BG_m)_k, G_k)$, where $G_k$ is the base change of $G$ to $k$, and we want $\nu_\alpha$ to be zero for all geometric points $\alpha$. One can show that if $S$ is connected then it suffices to check the condition $\nu_\alpha=0$ for a single $\alpha$.

(ii) $H^1((\BG_m)_k, G_k)$ is the set of $G(k)$-conjugacy classes of splittings for the canonical epimorphism $(\BG_m)_k\ltimes G_k\epi (\BG_m)_k$. The latter set depends only on the quotient of $G_k$ by its unipotent radical  (together with the action of $\BG_m$ on this quotient).

(iii) The definition of $\BT_n^G(S)$ from \S\ref{sss:BT_n^G via syntomification} makes sense without assuming the action of $\BG_m$ on $G$ to be 1-bounded. But the conjecture formulated below is \emph{unlikely to hold without this assumption.}

\subsection{The conjecture}  
For each $n\in\BN$ and $G\in\Shim_n$ we have defined a contravariant functor $S\mapsto \BT_n^G(S)$, where $S$ is a derived $p$-adic formal scheme of finite type over $\Spf\BZ_p$.

\begin{conj}   \label{conj:algebraicity}
For each $m\in\BN$ the restriction of this functor to the category of derived schemes over $\BZ/p^m\BZ$ is a quasicompact smooth 
algebraic stack\footnote{Here the words ``algebraic stack'' are understood in the derived sense (since $S$ is allowed to be derived); see the definition of 1-stack in \cite[Def.~5.1.3]{Lu} (which goes back to \cite{TV}).} over $\BZ/p^m\BZ$ with affine diagonal. 
\end{conj}

In particular, the conjecture would imply that the restriction of the functor $\BT_n^G$ to the category of classical schemes over $\BZ/p^m\BZ$ is a smooth algebraic stack in the sense of \cite{LM}.

As far as I understand, Conjecture~\ref{conj:algebraicity} has already been proved in \cite{GMM}.

\subsubsection{Remarks} \label{sss:Good news}
(i) If $S$ is a smooth $\BF_p$-scheme then the definitions of $\BT_1^G(S)$ given in \S\ref{sss:the key definition} and \S\ref{sss:BT_n^G via syntomification} are equivalent, see \cite[Thm.~3.8]{Shen} or Appendix~\ref{s:comparing the definitions of BT_1} below.

(ii) Let $G$ be the group $\BG_a\otimes\BZ/p^n\BZ$ equipped with the usual $\BG_m$-action. Then the triviality condition from \S\ref{sss:BT_n^G via syntomification} is automatic by \S\ref{sss:H^1(G_m, G)}, so $\BT_n^G(S)$ is the $\infty$-groupoid of all $\cO\{ 1\}$-torsors on $S^{\Syn}\otimes\BZ/p^n\BZ$. Bhatt and Lurie proved\footnote{Theorem~7.5.6 of \cite{BL} identifies the $p$-adic completion of $R\Gamma (S_{\et}, \cO_S^\times )[-1]$
with a certain complex $R\Gamma_{\syn}(S,\BZ_p(1))$, whose definition does not involve $S^{\Syn}$. If $S$ is quasisyntomic then it is known that $R\Gamma_{\syn}(S,\BZ_p(1))=R\Gamma (S^{\Syn},\cO\{ 1\})$.} that 
\begin{equation}   \label{e:cohomology of O{1}}
R\Gamma (S^{\Syn}\otimes\BZ/p^n\BZ,\cO\{ 1\})=R\Gamma (S_{\et}, \Cone (\cO_S^\times\overset{p^n}\longrightarrow\cO_S^\times)[-1]).
\end{equation}
if $S$ is quasisyntomic. Moreover, \eqref{e:cohomology of O{1}} is expected to hold\footnote{Again, the problem is to show that $R\Gamma_{\syn}(S,\BZ_p(1))=R\Gamma (S^{\Syn},\cO\{ 1\})$.} for \emph{any} $p$-nilpotent derived scheme of finite type over $\Spf\BZ_p$ (similarly to \cite[Prop.~8.16]{BL2}).
If so then $\BT_n^G$ identifies with the classifying stack of~$\mu_{p^n}$. This agrees with Conjecture~\ref{conj:algebraicity}; in the case $n=m=1$ this also agrees with \S\ref{sss:G=G_a}.

\section{Equivalence between the two definitions of $\BT_1^G(S)$}  \label{s:comparing the definitions of BT_1}
Let $G$ be a smooth affine group scheme over $\BF_p$ equipped with an action of $\BG_m$. Given a smooth $\BF_p$-scheme $S$, we have two definitions of the groupoid $\BT_1^G(S)$: an elementary one (see \S\ref{sss:the key definition}) and a definition via the stack
$S^{\Syn}$ (see \S\ref{sss:BT_n^G via syntomification}); both definitions make sense without assuming that the $\BG_m$-action is 1-bounded, see \S\ref{sss:without 1-boundedness} and \S\ref{sss:H^1(G_m, G)}(iii).
In this Appendix we sketch a proof of the equivalence between the two definitions.

\subsection{Recollections on $S^{\Syn}\otimes \BF_p$}
We will follow \cite[Ch.~2]{B}.
\subsubsection{}   \label{sss:from Nygaardization to syntomification}
As mentioned in \S\ref{sss:S^Syn}, $S^{\Syn}$ is obtained from $S^\cN$ by gluing together the two open substacks canonically isomorphic $S^\prism$. One has $S^\prism\otimes\BF_p=S^{\dR}:=S^{\dR /\BF_p}$.
Following \S 2.8 of~\cite{B}, we use the notation $S^C:=S^\cN\otimes\BF_p$. The stack $S^C$ contains two open substacks canonically isomorphic to $S^{\dR}$, and after gluing them together, one gets $S^{\Syn}\otimes\BF_p$.

Let us now recall the material on $S^C$ from \cite[Ch.~2]{B}.

\subsubsection{}    \label{sss:C}
Just as in Construction~2.8.2 from \cite{B}, let $C$ denote the quotient of the coordinate cross $\Spec\BF_p[u,t]/(ut)$ by the hyperbolic action of $\BG_m$.
One has $(\Spec\BF_p)^C=C$. The two irreducible components of $C$ are denoted by $C_{u=0}$ and $C_{t=0}$.

\subsubsection{}   \label{sss:the 2 components}
The preimage of $C_{t=0}$ (resp.~$C_{u=0}$) in $S^C$ is denoted in \cite{B} by $S^{\dR,c}$ (resp.~$S^{\dR,+}$); the superscript $c$ stands for ``conjugate filtration''.
Let $S^{\Hodge}:=S^{\dR,c}\cap S^{\dR,+}$.

In \S\ref{sss:from Nygaardization to syntomification} we mentioned two open substacks of $S^C$ canonically isomorphic $S^{\dR}$; these substacks are 
$S^{\dR,c}\setminus S^{\Hodge}$ and $S^{\dR,+}\setminus S^{\Hodge}$.

\subsubsection{}   \label{sss:pushout diagram}
By \S\ref{sss:from Nygaardization to syntomification} and \S\ref{sss:the 2 components}, we have a pushout diagram of stacks
\begin{equation}   \label{e:pushout diagram}
\xymatrix{
S^{\Syn}\otimes\BF_p& S^{\dR,+}\ar[l]\\
S^{\dR,c}\ar[u]\ & S^{\dR}\sqcup S^{\Hodge}\ar[l]\ar[u]
}
\end{equation}
Chapter 2 of \cite{B} contains a rather explicit description of the whole diagram \eqref{e:pushout diagram} via the procedure of \emph{transmutation} (see \cite[Rem.~2.3.8]{B}).
The key point is that it suffices to describe \eqref{e:pushout diagram} in the particular case $S=\BG_a$ \emph{as a diagram of ring stacks}. For such a description, see  \cite[2.8.3]{B} and references therein, as well as
\cite[2.5.1]{B}, \cite[2.7.8]{B}, and   \cite[Prop.~2.7.12]{B}. We paraphrase this description in \S\ref{ss:paraphrase}.

\subsection{$\BT_1^G(S)$ as a fiber product}
Let $\BT_1^G(S)$ denote the groupoid defined in \S\ref{sss:BT_n^G via syntomification}.
Thus $\BT_1^G(S)$ is the groupoid of $G_{\cO\{ 1\}}$-torsors $\cE$ on $S^{\Syn}\otimes\BF_p$ such that for every geometric point $\alpha :\Spec k\to S$, the pullback of $\cE$ via a certain morphism
$f_\alpha :\Spec k\times B\BG_m\to S^{\Syn}\otimes\BF_p$ is trivial. Let us note that $f_\alpha$ factors as
\begin{equation}   \label{e:factorizing f_alpha}
\Spec k\times B\BG_m\to S^{\Hodge}\to S^{\Syn}\otimes\BF_p,
\end{equation}
see the last paragraph of \S\ref{sss:S times BG_m to S^N}

Diagram \eqref{sss:pushout diagram} yields a pullback diagram of groupoids
\begin{equation}   \label{e:pullback diagram}
\xymatrix{
\BT_1^G(S)\ar[r] \ar[d] & \sX^{\dR,+}(S)\ar[d]\\
\sX^{\dR,c}(S)\ar[r] & \sX^{\dR}(S)\times \sX^{\Hodge}(S)
}
\end{equation}
where $\sX^{\dR}(S)$ is the groupoid of $G$-torsors\footnote{A $G_{\cO\{ 1\}}$-torsor on $S^{dR}$ is the same as a $G$-torsor because $(\Spec\BF_p)^{dR}=\Spec\BF_p$.} on $S^{dR}$ and 
$\sX^{\dR,+}(S)$ (resp.~$\sX^{\dR,c}(S)$ or $\sX^{\Hodge}(S)$) is the groupoid of $G_{\cO\{ 1\}}$-torsors on $S^{\dR,+}$ (resp.~$S^{\dR,c}$ or $S^{\Hodge}$) satisfying the triviality condition\footnote{The formulation of this condition uses the factorization \eqref{e:factorizing f_alpha}.} similar to the one from the definition of $\BT_1^G(S)$.

In the next subsection we will translate the diagram
\begin{equation}   \label{e:part of pullback diagram}
\sX^{\dR,c}(S)\to \sX^{\dR}(S)\times \sX^{\Hodge}(S)\leftarrow\sX^{\dR,+}(S)
\end{equation}
into an elementary language. The isomorphism between $\BT_1^G(S)$ and the groupoid from  \S\ref{sss:the key definition} 
will immediately follow from this description because diagram \eqref{e:pullback diagram} is Cartesian.

\subsection{Diagram \eqref{e:part of pullback diagram} in elementary terms}
Let $\fg:=\Lie (G)$. 
We will use the subgroups $M,P^\pm\subset G$ defined in \S\ref{ss:parabolics,Levi} and the grading $\fg=\bigoplus\limits_i\fg_i$ corresponding to the $\BG_m$-action.
 
\subsubsection{}
$\sX (S)$ is the groupoid of $G$-torsors on $S$ equipped with a nilpotent integrable connection.

\subsubsection{}   \label{sss:X Hodge}
$\sX^{\Hodge}(S)$ is the groupoid of pairs $(\cF_M ,\sigma)$, where $\cF_M$ is an $M$-torsor on $S$ and $\sigma\in H^0(S,(\fg_1)_{\cF_M}\otimes\Omega^1_S)$,
see \cite[Rem.~2.5.9]{B}.

\subsubsection{}   \label{sss:X dR+}
$\sX^{\dR,+}(S)$ is the groupoid of $P^-$-torsors $\cF^-$ on $S$ equipped with a nilpotent integrable connection $\nabla$ on the corresponding $G$-bundle $\cF^-_G$ satisfying the Griffiths transversality condition
$\KS_\nabla\in H^0(S,(\fg_1)_{\cF^-}\otimes\Omega^1_S)\subset H^0(S,(\fg/\fg_{\le 0})_{\cF^-}\otimes\Omega^1_S)$ (we are using the notation from \S\ref{sss:defining KS}).
See \cite[Rem.~2.5.8]{B}.

\subsubsection{}   \label{sss:X dR,c}
$\sX^{\dR,c}(S)$ is the groupoid of $P^+$-torsors $\cF^+$ on $S$ equipped with an integrable connection $\nabla$ inducing a $p$-integrable connection on the $M$-torsor $\cF^+_M$ corresponding to $\cF^+$.
See \cite[Rem.~2.7.10]{B}.

\subsubsection{}  \label{sss:interesting functor}
The functor $\sX^{\dR,c}(S)\to\sX^{\Hodge}(S)$ from diagram \eqref{e:part of pullback diagram} is as follows:

(i) the $M$-torsor $\cF_M$ from \S\ref{sss:X Hodge} is the $\Fr_S$-descent of the $M$-torsor $\cF^+_M$ from \S\ref{sss:X dR,c} via the $p$-integrable connection;

(ii) $\sigma\in H^0(S,(\fg_1)_{\cF_M}\otimes\Omega^1_S)$ is the image of $-\pCurv_\nabla\in H^0(S,(\fg_{\ge 1})_{\cF_M}\otimes\Omega^1_S)$.

The other functors from diagram \eqref{e:part of pullback diagram} are self-explanatory.

\subsubsection{Origin of the sign}  
The ``minus'' sign in \S\ref{sss:interesting functor}(ii) comes from the ``minus'' sign in formula~\eqref{e:-V} below.

\subsection{The ring stack $\BG_a^{\Syn}\otimes\BF_p$}    \label{ss:paraphrase}
Let $\BG_a$ denote the additive group over $\BF_p$. In this subsection we recall the description of $\BG_a^{\Syn}\otimes\BF_p$ given in \cite[Ch.~2]{B} and explain the
origin of the ``minus'' sign in \S\ref{sss:interesting functor}(ii).

\subsubsection{Plan}
$\BG_a$ is a ring scheme, so $\BG_a^{\Syn}\otimes\BF_p$ is a ring stack over $(\Spec\BF_p)^{\Syn}\otimes\BF_p$. One gets $(\Spec\BF_p)^{\Syn}\otimes\BF_p$ from the
stack $C$ considered in \S\ref{sss:C} by gluing together the two open points of $C$. Equivalently, $(\Spec\BF_p)^{\Syn}\otimes\BF_p$ is obtained from $\BP^1/\BG_m$ by gluing
$\{ 0\}/\BG_m$ with $\{ \infty\}/\BG_m$.

Let $\sR$ be the pullback of $\BG_a^{\Syn}\otimes\BF_p$ via the map $\BP^1\to(\Spec\BF_p)^{\Syn}\otimes\BF_p$. To describe $\BG_a^{\Syn}\otimes\BF_p$, we will describe in \S\ref{sss:describing sR}-\ref{sss:identifying the fibers over 0 and infty}
the $\BG_m$-equivariant ring stack $\sR$ over $\BP^1$ and the isomorphism $\sR_0\iso\sR_\infty$, where $\sR_0,\sR_\infty$ are the fibers of $\sR$ over $0,\infty\in\BP^1$.

\subsubsection{The ring stack $\sR$}  \label{sss:describing sR}
Let $(\BG_a)_{\BP^1}=\BG_a\times\BP^1$; this is a ring scheme over $\BP^1$. One has 
$$\sR =\Cone (G\overset{d}\longrightarrow (\BG_a)_{\BP^1}),$$
where $(G,d)$ is a quasi-ideal\footnote{For the language of quasi-ideals and cones, see \cite[\S 1.3.3-1.3.4]{Prismatization}.} in $(\BG_a)_{\BP^1}$, which we are going to define.

Let $\BG_a^\sharp$ denote the the PD-hull of zero in $\BG_a$. Then $\BG_a^\sharp$ is a $\BG_a$-module, and the natural map $f:\BG_a^\sharp\to\BG_a$ is a $\BG_a$-module homomorphism such that
$$\im f=\alpha_p:=\Ker (\Fr :\BG_a\to\BG_a).$$
One defines $G$ to be a certain $\BG_a$-submodule of $\BG_a^\sharp\times\BG_a\times\BP^1$ and $d:G\to\BG_a\times\BP^1$ to be the projection. The equations defining $G\subset\BG_a^\sharp\times\BG_a\times\BP^1$ are as follows:
\begin{equation}  \label{e:uy=f(x)}
uy=f(x),
\end{equation}
\begin{equation}   \label{e:Fr(y)=0}
y^p=0,
\end{equation}
where $x\in\BG_a^\sharp$, $y\in\BG_a$, $u\in\BP^1$. Strictly speaking, \eqref{e:uy=f(x)} means that
\[
uy=f(x) \mbox{ if }u\ne\infty, \quad y=u^{-1}f(x) \mbox{ if }u\ne 0.
\]
Note that if $u\ne 0$ then \eqref{e:Fr(y)=0} follows from \eqref{e:uy=f(x)}.

Finally, the action of $\lambda\in\BG_m$ on $G$ and $\BG_a\times\BP^1$ is given by $\tilde x=\lambda^{-1} x$, $\tilde u=\lambda^{-1} u$, $\tilde y=y$.

\subsubsection{The isomorphism $\sR_0\iso\sR_\infty$}   \label{sss:identifying the fibers over 0 and infty}
Let $K:=\Ker (\BG_a^\sharp\overset{f}\longrightarrow\BG_a)$. Then
$$ \quad \sR_0=\Cone (K\oplus\alpha_p\overset{(0,1)}\longrightarrow\BG_a)=\Cone (K\overset{0}\longrightarrow\BG_a/\alpha_p)=\Cone (K\overset{0}\longrightarrow\BG_a)$$
(we have used the ring isomorphism $\BG_a/\alpha_p\iso\BG_a$ induced by $\Fr:\BG_a\to\BG_a$). On the other hand,
$$\sR_\infty=\Cone (\BG_a^\sharp\overset{0}\longrightarrow\BG_a).$$
The isomorphism $\sR_0\iso\sR_\infty$ comes from the isomorphism
\begin{equation}    \label{e:-V} 
-V:\BG_a^\sharp\iso K,
\end{equation}
where $V$ is the Verschiebung of $\BG_a^\sharp$. 

\subsubsection{Origin of formula \eqref{e:-V}}  \label{sss:Origin}
One gets \eqref{e:-V} by comparing \cite[\S 2.8.3]{B} with \cite[\S 2.8.1]{B} and \cite[Prop.~ 2.7.12]{B}.
The ``minus'' sign in \eqref{e:-V} comes from diagram (2.7.5) of \cite{B}.

\bibliographystyle{alpha}

\end{document}